\documentclass{amsart}

\usepackage{graphicx,epsfig}

\usepackage{amsmath,amssymb,latexsym, amsfonts, amscd, amsthm, xy}

\usepackage{url}

\usepackage{hyperref}

\input{xy}
\xyoption{all}

\usepackage[usenames]{color}

\makeindex 

=630pt \headheight = 12pt \headsep = 20pt

\hypersetup{
    colorlinks   = true,
      citecolor    = red
}

\hypersetup{linkcolor=red}

\hypersetup{linkcolor=blue}

\theoremstyle{plain}
\newtheorem{theorem}{Theorem}[section]

\newtheorem{proposition}[theorem]{Proposition}
\newtheorem{corollary}[theorem]{Corollary}

\newtheorem{lemma}[theorem]{Lemma}

\newtheorem{question}[theorem]{Question}

\theoremstyle{definition}

\newtheorem{remark}[theorem]{Remark}

\def\bA{\mathbb{A}}

\def\bF{\mathbb{F}}

\def\bQ{\mathbb{Q}}
\def\bZ{\mathbb{Z}}

\def\A{\mathbf{A}}

\def\S{\mathbf{S}}

\def\cE{\mathcal{E}}
\def\cF{\mathcal{F}}

\def\cG{\mathcal{G}}

\def\cM{\mathcal{M}}

\def\cO{\mathcal{O}}

\def\cP{\mathcal{P}}

\def\cS{\mathcal{S}}

\def\cX{\mathcal{X}}

\def\cU{\mathcal{U}}

\def\fR{\mathfrak{R}}

\def\deg{\mathbf{deg}}

\def\SL{\mathbf{SL}}

\def\UM{\mathbf{UM}}

\def\SR{\mathbf{SR}}

\begin{document}

\title{Certain sets over function fields are polynomial families}

\author{Nguyen Ngoc Dong Quan}

\date{November 3, 2015}

\address{Department of Mathematics \\
         The University of Texas at Austin \\
         Austin, TX 78712 \\
         USA}

\email{\href{mailto:dongquan.ngoc.nguyen@gmail.com}{\tt dongquan.ngoc.nguyen@gmail.com}}

\maketitle

\tableofcontents

\begin{abstract}

In 1938, Skolem conjectured that $\SL_n(\bZ)$ is not a polynomial family for any $n \ge 2$. Carter and Keller disproved Skolem's conjecture for all $n \ge 3$ by proving that $\SL_n(\bZ)$ is boundedly generated by the elementary matrices, and hence a polynomial family for any $n \ge 3$. Only recently, Vaserstein refuted Skolem's conjecture completely by showing that $\SL_2(\bZ)$ is a polynomial family. An immediate consequence of Vaserstein's theorem also implies that $\SL_n(\bZ)$ is a polynomial family for any $n \ge 3$. In this paper, we prove a function field analogue of Vaserstein's theorem: that is, if $\A$ is the ring of polynomials over a finite field of odd characteristic, then $\SL_2(\A)$ is a polynomial family in 52 variables. A consequence of our main result also implies that $\SL_n(\A)$ is a polynomial family for any $n \ge 3$. 

\end{abstract}

\section{Introduction}

Let $\fR$ be a commutative ring with identity, and let $\cX$ be a subset of $\fR^h$. The set $\cX$ is said to be a \textit{polynomial family over $\fR$ with $d$ parameters} for some positive integer $d$ if there exist polynomials $\cP_1, \ldots, \cP_h \in \fR[x_1, \ldots, x_d]$ in $d$ variables $x_1, \ldots, x_d$ such that 
\begin{align*}
\cX = \cP(\fR^d),
\end{align*}
where $\cP$ is the polynomial map in $d$ variables $x_1, \ldots, x_d$ of the form
\begin{align*}
\cP(x_1, \ldots, x_d) = (\cP_1(x_1, \ldots, x_d), \ldots, \cP_h(x_1, \ldots, x_d)).
\end{align*}
We also say that $\cP$ is a \textit{polynomial parametrization of $\cX$}.

Determining whether a set in $\fR^h$ is a polynomial family has a long history dating back to the 17th century. For example, when $\fR = \bZ$, Lagrange's four-square theorem, n\'ee Bachet's conjecture, states that every nonnegative integer can be represented as the sum of four integer squares. Equivalently, the theorem says that the set $\bZ_{\ge 0}$ of nonnegative integers is a polynomial family with 4 parameters, and the polynomial $\cP \in \bZ[x_1, x_2, x_3, x_4]$ defined by
\begin{align*}
\cP(x_1, x_2, x_3, x_4) = x_1^2 + x_2^2 + x_3^2 + x_4^2
\end{align*}
is a polynomial parametrization of $\bZ_{\ge 0}$. 

In \cite[page {\bf23}]{Skolem}, Skolem conjectured that $\SL_n(\bZ)$ is not a polynomial family for any $n \ge 2$. Carter and Keller \cite{Carter-Keller} disproved this for all $n \ge 3$ by proving that $\SL_n(\bZ)$ is \textit{boundedly generated by the elementary matrices} for each $n \ge 3$, and thus is a polynomial family for each $n \ge 3$.

Recall that a group $\cG$ is said to be \textit{boundedly generated by a subset $\Gamma$ of $\cG$} if there exists a positive integer $\ell$ such that every element $g \in \cG$ can be written in the form
\begin{align*}
g = \gamma_1\ldots \gamma_r,
\end{align*}
where $r \le \ell$, and the $\gamma_i$ are elements of $\Gamma \cup \Gamma^{-1}$. We further say that $\cG$ is \textit{boundedly generated by the elementary matrices} if $\Gamma$ is the set of elementary matrices. 

When $\Gamma$ is invariant under the map $\phi : \cG \rightarrow \cG, g \mapsto g^{-1}$ and contains the identity element of $\cG$, then $\Gamma \cup \Gamma^{-1} = \Gamma$, and one can write
\begin{align*}
\cG = \underbrace{\Gamma \cdot \Gamma \cdots \Gamma}_\text{{$r$ copies of $\Gamma$}}.
\end{align*}

When $\Gamma$ is a finite union of cyclic groups, the above definition leads to the following: \textit{$\cG$ is said to have bounded generation if there exist cyclic subgroups $\Gamma_1, \ldots, \Gamma_h$ of $\cG$ for some integer $h \ge 1$ such that $\cG = \Gamma_1\Gamma_2\cdots \Gamma_h$.}

Fix an integer $n \ge 2$. For any integers $i, j$ with $1 \le i \ne j \le n$ and $\alpha \in \bZ$, let $\cE_{i, j}(\alpha)$ be the matrix in $\SL_n(\bZ)$ such that all the entries on the diagonal are $1$, the $(i, j)$ entry is $\alpha$, and all other entries are $0$. Set 
\begin{align*}
\cE_{i, j}(\bZ) = \{\cE_{i, j}(\alpha) \; | \; \alpha \in \bZ\}.
\end{align*}
It is known that the $\cE_{i, j}(\bZ)$ are cyclic subgroups of $\SL_n(\bZ)$, and the set of elementary matrices in $\SL_n(\bZ)$ is a union of the $\cE_{i, j}(\bZ)$. Thus for a fixed integer $n \ge 2$, if  $\SL_n(\bZ)$ is boundedly generated by the elementary matrices, then $\SL_n(\bZ)$ has bounded generation. For example, Carter--Keller's theorem \cite{Carter-Keller} implies that $\SL_n(\bZ)$ has bounded generation for all $n \ge 3$.

It is well-known that $\SL_2(\bZ)$ is finitely generated, but not boundedly generated by the elementary matrices since it has a free subgroup of index $12$. Indeed, assume the contrary, i.e., $\SL_2(\bZ)$ is boundedly generated by the elementary matrices. It then follows from the above discussion that $\SL_2(\bZ)$ has bounded generation. The next result is well-known, and its proof can be found, for example, in \cite[Proposition 1.1]{Heald} .

\begin{proposition}
\label{Proposition-bounded-generation-of-finite-index-subgroups}

Let $\cG$ be a group, and $\cS$ be a subgroup of $\cG$ such that $[\cG : \cS]$ is finite. Then $\cG$ has bounded generation if and only if $\cS$ has bounded generation.

\end{proposition}

Let $\cS$ be the subgroup of $\SL_2(\bZ)$ that is generated by $\begin{pmatrix} 1 & 2 \\ 0 & 1\end{pmatrix}$ and $\begin{pmatrix} 1 & 0 \\ 2 & 1 \end{pmatrix}$. Sanov's theorem tells us that $\cS$ is a free group, and $[\SL_2(\bZ) : \cS] = 12$. By Proposition \ref{Proposition-bounded-generation-of-finite-index-subgroups}, $\cS$ has bounded generation, which is a contradiction since nonabelian free groups do not have bounded generation (see Tavgen' \cite{Tavgen}). Thus $\SL_2(\bZ)$ is not boundedly generated by the elementary matrices. It is worth mentioning here that Tavgen' \cite{Tavgen} even proved that if $F$ is either $\bQ$ or an imaginary quadratic field over $\bQ$, then the elementary matrices do not boundedly generate $\SL_2(\cO_F)$, where $\cO_F$ is the ring of integers of $F$.

The above discussion implies that one cannot expect to use the same arguments as Carter and Keller \cite{Carter-Keller} to disprove Skolem's conjecture for $\SL_2(\bZ)$. In fact, only recently, Vaserstein \cite{Vaserstein} refuted Skolem's conjecture completely by proving that $\SL_2(\bZ)$ is a polynomial family with $46$ parameters. As an immediate consequence, Vaserstein also showed that $\SL_n(\bZ)$ with $n \ge 3$ is a polynomial family with less parameters than in the work of Carter and Keller \cite{Carter-Keller}. Following the work of Vaserstein, it is not difficult to show that for a commutative ring $\fR$ satisfying the second Bass stable range condition (see Bass \cite{Bass} for this definition), if $\SL_2(\fR)$ is a polynomial family, then so is $\SL_n(\fR)$ for any $n \ge 3$. It is well-known (see Bass \cite{Bass}) that every Dedekind domain satisfies the second Bass stable range condition, and hence for such a domain $\fR$, it suffices to consider whether $\SL_2(\fR)$ is a polynomial family.

Now return to a general setting in which we fix a commutative ring $\fR$ with identity. The question as to whether $\SL_2(\fR)$ is a polynomial family can be rephrased in terms of the solutions of a Diophantine equation as follows. One can realize $\SL_2$ as a hypersurface in $\bA^4$ by
\begin{align}
\label{E-Diophantine-eqn-of-SL_2-in-the-introduction}
x_1x_2 - x_3x_4 = 1.
\end{align}
Then $\SL_2(\fR)$ is a polynomial family if and only if all the $\fR$-integral solutions of (\ref{E-Diophantine-eqn-of-SL_2-in-the-introduction}) can be obtained from a fixed polynomial parametrization with coefficients in $\fR$ by letting all the variables run through $\fR$. For example, Vaserstein's theorem says that all the integral solutions of (\ref{E-Diophantine-eqn-of-SL_2-in-the-introduction}) can be obtained from a fixed polynomial parametrization with $\bZ$-coefficients in 46 parameters by letting all the variables run through $\bZ$. 

It is natural to consider the solutions of a Diophantine equation in a more general ring than the ring $\bZ$ of integers. In this direction, it is natural to extend Vaserstein's theorem to a ring of integers in a number field or a function field. Before discussing related results in this direction, let us fix some notation. 

For each $h \ge 1$, we denote by $\cF_{2h}(m^{(1)}, \ldots, m^{(2h)}) \in \SL_2(\bZ[m^{(1)}, \ldots, m^{(2h)}])$ the polynomial matrix in $2h$ parameters defined by
\begin{align*}
\cF_{2h}(m^{(1)}, \ldots, m^{(2h)}) = \begin{pmatrix} 1 & m^{(1)} \\ 0 & 1 \end{pmatrix}\begin{pmatrix} 1 & 0 \\ m^{(2)} & 1 \end{pmatrix}\cdots \begin{pmatrix} 1 & m^{(2h - 1)} \\ 0 & 1 \end{pmatrix}\begin{pmatrix} 1 & 0 \\ m^{(2h)} & 1 \end{pmatrix}.
\end{align*}
Similarly, for each integer $h \ge 0$, we denote by $\cF_{2h + 1}(m^{(1)}, \ldots, m^{(2h + 1)}) \in \SL_2(\bZ[m^{(1)}, \ldots, m^{(2h + 1)}])$ the polynomial matrix in $2h + 1$ parameters defined by
\begin{align*}
\cF_{2h + 1}(m^{(1)}, \ldots, m^{(2h + 1)}) = \begin{pmatrix} 1 & m^{(1)} \\ 0 & 1 \end{pmatrix}\begin{pmatrix} 1 & 0 \\ m^{(2)} & 1 \end{pmatrix}\cdots \begin{pmatrix} 1 & m^{(2h - 1)} \\ 0 & 1 \end{pmatrix}\begin{pmatrix} 1 & 0 \\ m^{(2h)} & 1 \end{pmatrix}\begin{pmatrix} 1 & m^{(2h + 1)} \\ 0 & 1 \end{pmatrix}.
\end{align*}

Since the $\cF_n$ are defined over the integers, one can view the $\cF_n$ as elements in $\SL_2(\fR[m^{(1)}, \ldots, m^{(n)}])$, where $\fR$ is a commutative ring with $1$. When $\fR$ is the polynomial ring in one variable with coefficients in the finite field of $q$ elements, the definition of the $\cF_n$ agrees with that of the matrix maps, also denoted by $\cF_n$, in Subsection \ref{SS-F_h--and--G_h}.

In 1996, Zannier \cite{Zannier} proved that conditionally under the truth of the Generalized Riemann Hypothesis, $\cF_5$ is surjective over $\bZ[\sqrt{2}]$, which implies that $\SL_2(\bZ[\sqrt{2}])$ is a polynomial family with 5 parameters. In 2003, Zannier \cite{Zannier-2003} unconditionally showed that $\cF_5$ is surjective over $\cO_S$, which implies that $\SL_2(\cO_S)$ is a polynomial family with 5 parameters. Here $S = \{2, 3, \wp\}$ with $\wp$ being a prime such that $\wp \equiv 1 \pmod{4}$, and $\cO_S$ is the ring of $S$-integers in $\bQ$ defined by
\begin{align}
\label{Zannier-set-introduction}
\cO_S = \{q \in \bQ \; | \; \text{there exist nonnegative integers $\alpha_2, \alpha_3, \alpha_{\wp}$ such that $q2^{\alpha_2}3^{\alpha_3}\wp^{\alpha_{\wp}} \in \bZ$}\}.
\end{align}

Before discussing the work of Zannier in more detail, let us digress a moment to explain the relation between the maps $\cF_n$ and continued fractions. 

Let $\lambda_1, \ldots, \lambda_m$ be real numbers such that $\lambda_i \ge 1$ for each $2 \le i \le m$. For each integer $1 \le h \le m$, the symbol $[\lambda_1, \ldots, \lambda_h]$ is defined recursively by $[\lambda_1] = \lambda_1$, and $[\lambda_1, \lambda_2, \ldots, \lambda_h] = \lambda_1 + \dfrac{1}{[\lambda_2, \ldots, \lambda_h]}$. (Note that $[\lambda_2, \ldots, \lambda_h] > 0$ by the assumption that $\lambda_i \ge 1$ for each $i \ge 2$; so the symbol $[\lambda_1, \ldots, \lambda_h]$ is well-defined.) 

Assume now that the $\lambda_i$ lie in $\bZ$. Then the symbol $[\lambda_1, \ldots, \lambda_h]$ is the \textit{$h$-th convergent of the continued fraction}. It is clear that the $h$-th convergent is a rational number.

For each $1 \le h \le m$, let
\begin{align*}
\dfrac{P_h}{Q_h} = [\lambda_1, \ldots, \lambda_h],
\end{align*}
where $P_h, Q_h$ are relatively prime integers. We are only interested in a special case when $m = 2n$ for some positive integer $n$. If we let $P_0 = Q_{-1} = 1$ and $P_{-1} = Q_0 = 0$, then it is known that $P_h = \lambda_h P_{h - 1} + P_{h - 2}$ and $Q_h = \lambda_h Q_{h - 1} + Q_{h - 2}$ for each $h \ge 1$. Since $P_{2n}Q_{2n - 1} - P_{2n - 1}Q_{2n} = (-1)^{2n} = 1$, there exists an integer $\mu$ such that
\begin{align*}
\begin{pmatrix} 1 & 0 \\ -\lambda_1 & 1 \end{pmatrix}\begin{pmatrix} 1 & -\lambda_2 \\ 0 & 1 \end{pmatrix} \cdots \begin{pmatrix} 1 & 0 \\ -\lambda_{2n - 1} & 1 \end{pmatrix} \begin{pmatrix} 1 & -\lambda_{2n} \\ 0 & 1 \end{pmatrix}\begin{pmatrix} P_{2n} & Q_{2n} \\ P_{2n - 1} & Q_{2n - 1} \end{pmatrix} = \begin{pmatrix} 1 & \mu \\ 0 & 1 \end{pmatrix},
\end{align*}
and thus
\begin{align*}
\begin{pmatrix} P_{2n} & Q_{2n} \\ P_{2n - 1} & Q_{2n - 1} \end{pmatrix} = \begin{pmatrix} 1 & \lambda_{2n} \\ 0 & 1 \end{pmatrix}\begin{pmatrix} 1 & 0 \\ \lambda_{2n - 1} & 1 \end{pmatrix}\cdots \begin{pmatrix} 1 & \lambda_2 \\ 0 & 1 \end{pmatrix}\begin{pmatrix} 1 & 0 \\ \lambda_1 & 1 \end{pmatrix}\begin{pmatrix} 1 & \mu \\ 0 & 1 \end{pmatrix} = \cF_{2n + 1}(\lambda_{2n},\ldots, \lambda_1, \mu).
\end{align*}

When $n = 2$, Zannier \cite{Zannier} proved that the Generalized Riemann Hypothesis implies the surjectivity of the matrix map on the left-hand side of the above equation over $\bZ[\sqrt{2}]$. Hence $\cF_5$ is surjective over $\bZ[\sqrt{2}]$ by the above equation. Replaced $\bZ[\sqrt{2}]$ by $\cO_S$ with $\cO_S$ defined by (\ref{Zannier-set-introduction}), Zannier \cite{Zannier-2003} unconditionally obtained the same results.

In 2007, Morris \cite{Morris} provided details of a proof of the next result that is contained in an unpublished work of Carter, Keller, and Paige \cite{Carter-Keller-Paige}.

\begin{theorem}
\label{Theorem-Carter-Keller-Paige}
$(\text{Carter--Keller--Paige, see \cite{Carter-Keller-Paige} and \cite[Theorem 1.2]{Morris}})$

Let $F$ be a number field, and let $\cO_F$ be the ring of integers of $F$. Then $\SL_2(\cO_F)$ is boundedly generated by the elementary matrices if and only if $\cO_F$ has infinitely many units.

\end{theorem}

The Carter--Keller--Paige theorem provides a very large class of rings $\fR$ for which $\SL_2(\fR)$ is a polynomial family. In fact, the theorem proves that if $F$ is a number field such that $F \ne \bQ$ and $F$ is not an imaginary quadratic field over $\bQ$, then $\SL_2(\cO_F)$ is a polynomial family. Hence, in the number field setting, it remains to consider whether or not $\SL_2(\cO_F)$ is a polynomial family when $F$ is an imaginary quadratic field. 

Since $\bZ[\sqrt{2}]$ contains infinitely many units, Theorem \ref{Theorem-Carter-Keller-Paige} also provides a group-theoretic proof of Zannier's result that $\SL_2(\bZ[\sqrt{2}])$ is a polynomial family. Although Theorem \ref{Theorem-Carter-Keller-Paige} provides an unconditional proof of a corollary of Zannier's result, it does not give an explicit bound for the number of parameters as obtained in the work of Zannier \cite{Zannier}. 

The bounded number of elementary matrices needed to generate $\SL_2(\cO_F)$ in Theorem \ref{Theorem-Carter-Keller-Paige} depends on the Compactness Theorem in Model Theory (see \cite[Theorem 2.1.4]{Marker}). Thus Theorem \ref{Theorem-Carter-Keller-Paige} does not provide any explicit bound on the number of elementary matrices. It is natural to ask the following questions.

\begin{question}

\begin{itemize}

\item []

\item [(i)] Let $F$ be a number field, and let $\cO_F$ be the ring of integers of $F$. Assume that $\cO_F$ has infinitely many units. What is an explicit bound for the number of elementary matrices needed to generate $\SL_2(\cO_F)$?

\item [(ii)] Does there exist a positive integer $\cU$ such that for any number field $F$ with the ring of integers $\cO_F$ containing infinitely many units, the number of elementary matrices needed to generate $\SL_2(\cO_F)$ is less than $\cU$?

\end{itemize}

\end{question}

Let $F$ be a number field, and let $\cO_F$ be the ring of integers of $F$. Let $S$ be a finite set of primes of $F$ that contains all the Archimedean primes. Let $\cO_{S, F}$ be the ring of $S$-integers of $F$, and let $U_S$ denote the group of units in $\cO_{S, F}$. When $U_S$ is infinite, Cooke and Weinberger \cite{Cooke-Weinberger} proved that the Generalized Riemann Hypothesis implies that $\SL_2(\cO_{S, F})$ is a polynomial family with $9$ parameters. They further showed that if $F$ admits a real embedding, then $7$ parameters is sufficient. The results of Cooke-Weinberger \cite{Cooke-Weinberger} provide a conditional answer to the above questions under the truth of the Generalized Riemann Hypothesis. It is interesting if one can obtain another proof of Theorem \ref{Theorem-Carter-Keller-Paige} that does not use the Compactness Theorem in Model Theory. Such a proof should shed some light on the above questions from another viewpoint that may result in an unconditional answer to the above questions.

In Zannier \cite{Zannier} \cite{Zannier-2003}, the number of elementary matrices needed to generate $\SL_2(\bZ[\sqrt{2}])$ or $\SL_2(\cO_S)$ with $\cO_S$ defined by (\ref{Zannier-set-introduction}) is 5, which is quite small. It is possible, as remarked in Zannier \cite{Zannier-2003} that $5$ should be the smallest number of parameters needed over $\bZ[\sqrt{2}]$ or the rings $\cO_S$. This motivates the following question.

\begin{question}

If $\fR$ is a ring such that $\SL_2(\fR)$ is a polynomial family, what is the smallest number of parameters needed to polynomially parametrize $\SL_2(\fR)$? 

\end{question}

For each ring $\fR$ with $\SL_2(\fR)$ being a polynomial family, denote by $\cM(\fR)$ the smallest number parameters needed to polynomially parametrize $\SL_2(\fR)$. Then Theorem 1 in Zannier \cite{Zannier} shows that $\cM(\cO_K) \ge 4$ if $\cO_K$ is the ring of integers in a number field $K$. In particular, Vaserstein's theorem \cite{Vaserstein} combined with Zannier's theorem \cite{Zannier} imply that $4 \le \cM(\bZ) \le 46$. It is certainly interesting if one can find a precise value of $\cM(\fR)$, where $\fR$ is the ring of integers in a number field or a function field. 

Let $p$ be an \textit{odd} prime, and let $q$ be a power of $p$. Let $\A = \bF_q[T]$, where $\bF_q$ is the finite field with $q$ elements, and $T$ denotes an indeterminate. The main aim of this paper is to determine an upper bound for $\cM(\A)$; more precisely, our main goal in this paper is to prove the following.

\begin{theorem}
\label{Thm-the-1st-main-thm-in-the-introduction}
$(\text{See Theorem \ref{T-SL2(A)--is--a--polynomial-family--in--S--SL2(A)}})$

$\SL_2(\A)$ is a polynomial family with 52 parameters.

\end{theorem}

Despite many strong analogies between $\bZ$ and $\A$ (see Goss \cite{Goss}, Rosen \cite{Rosen}, Thakur \cite{Thakur-book}, or Weil \cite{Weil} for these analogies), $\SL_2(\A)$ does not always bear a resemblance to $\SL_2(\bZ)$. For example, Nagao's theorem (see Nagao \cite{Nagao}, or Bux and Wortman \cite[Section {\bf2}]{Bux-Wortman}) says that $\SL_2(\A)$ is not finitely generated. The group $\SL_2(\bZ)$ is however finitely generated as mentioned before. So it is a nontrivial question as to whether there is an analogue of Vaserstein's theorem for $\A$. Theorem \ref{Thm-the-1st-main-thm-in-the-introduction} answers this questions affirmatively by showing that $\SL_2(\A)$ is a polynomial family with 52 parameters.

Throughout the work of Vaserstein \cite{Vaserstein}, the polynomial parametrization of $\SL_2(\bZ)$ is often used to show many interesting sets in $\bZ^h$ are polynomial families. Using similar arguments as in Vaserstein \cite{Vaserstein}, one can use Theorem \ref{Thm-the-1st-main-thm-in-the-introduction} to show many sets in $\A^h$ are polynomial families. As an illustration, let us now consider some applications of Theorem \ref{Thm-the-1st-main-thm-in-the-introduction}. 

Take any commutative ring $\fR$ with identity $1$. Recall that a $h$-tuple $(m_1, \ldots, m_h) \in \fR^h$ is called \textit{unimodular} if there exist elements $\alpha_1, \ldots, \alpha_h \in \fR$ such that $\sum_{i = 1}^h\alpha_i m_i = 1$. We denote by $\UM_h(\fR)$ the set of all unimodular $h$-tuples in $\fR^h$. 

We say that $\fR$ \textit{satisfies the $h$-th Bass stable range condition} if for any $(h + 1)$-tuple $(m_1, \ldots, m_{h + 1}) \in \UM_{h + 1}(\fR)$, there exist elements $\alpha_1, \ldots, \alpha_h \in \A$ such that the $h$-tuple $(m_1 + \alpha_1 m_{h +1}, \ldots, m_h + \alpha_h m_{h + 1}) \in \UM_h(\fR)$. In notation, we write $\SR(\fR) \le h$.

Now return to our ring $\A$. It is well-known (see \cite[page {\bf14}]{Bass}) that $\A$ satisfies the second Bass stable range condition. Hence it follows from Vaserstein \cite[pages {\bf994} and {\bf995}]{Vaserstein} that $\UM_n(\A)$ is a polynomial family with $2n$ parameters for all $n \ge 3$. (In fact, Vaserstein proved that the last result also holds if $\A$ is replaced by any commutative ring $\fR$ with $\SR(\fR) \le 2$.) 

Now take any pair $(a, b) \in \UM_2(\A)$. Then there exist $c, d \in \A$ such that $ad - bc = 1$. Set 
\begin{align*}
\alpha = \begin{pmatrix} a & b \\ c & d \end{pmatrix} \in \SL_2(\A).
\end{align*}
By Theorem \ref{Thm-the-1st-main-thm-in-the-introduction}, there are polynomials $\cP_1, \cP_2, \cP_3, \cP_4 \in \A[x_1, \ldots, x_{52}]$ in 52 variables such that
\begin{align*}
\SL_2(\A) = \begin{pmatrix} \cP_1(\A^{52}) & \cP_2(\A^{52}) \\ \cP_3(\A^{52}) & \cP_4(\A^{52}) \end{pmatrix}.
\end{align*}

We deduce that
\begin{align*}
(a, b) = (1, 0)\alpha \in (1, 0)\SL_2(\A) = (1, 0) \begin{pmatrix} \cP_1(\A^{52}) & \cP_2(\A^{52}) \\ \cP_3(\A^{52}) & \cP_4(\A^{52}) \end{pmatrix} = ( \cP_1(\A^{52}) , \cP_2(\A^{52})),
\end{align*}
which yields the following result.

\begin{corollary}

$\UM_2(\A)$ is a polynomial family with 52 parameters.

\end{corollary}

Following the same arguments as in Vaserstein \cite[page {\bf998}]{Vaserstein} and using Theorem \ref{Thm-the-1st-main-thm-in-the-introduction}, the following result is immediate, and can be proved by induction on $n$.

\begin{corollary}

$\SL_n(\A)$ is a polynomial family with $45 + n(3n + 1)/2$ parameters for any $n \ge 2$.

\end{corollary}

The proof of Theorem \ref{Thm-the-1st-main-thm-in-the-introduction} in Section \ref{S-SL2(A)-polynomial-family} shows that there exists a surjective matrix map from $\A^{52}$ to $\SL_2(\A)$ with \textit{rational integral coefficients}. Hence Theorem \ref{Thm-the-1st-main-thm-in-the-introduction} also implies that $\SL_2(\bar{\bF}_q[T])$ is a polynomial family, where $\bar{\bF}_q$ is the algebraic closure of $\bF_q$. It is natural to ask the following question.

\begin{question}

Let $\mathfrak{F} : \bA^{52} \rightarrow \SL_2$ be the morphism defined over $\bZ$ that is constructed in the proof of Theorem \ref{Thm-the-1st-main-thm-in-the-introduction} in Section \ref{S-SL2(A)-polynomial-family}. Does there exist a ring $\fR$ such that the matrix map $\mathfrak{F}(\fR) : \bA^{52}(\fR) \rightarrow \SL_2(\fR)$ is not surjective?

\end{question}

One can ask the same question with $\mathfrak{F}$ replaced by the morphism $\mathfrak{V} : \bA^{48} \rightarrow \SL_2$ from the work of Vaserstein \cite{Vaserstein}. In fact we do not even know whether or not the matrix map $\mathfrak{V}(\A) : \bA^{48}(\A) \rightarrow \SL_2(\A)$ arising from the Vaserstein morphism is surjective.

\subsection{Main ideas of the proof of Theorem \ref{Thm-the-1st-main-thm-in-the-introduction}}

In this subsection, we explain the main ideas of the proof of Theorem \ref{Thm-the-1st-main-thm-in-the-introduction}. Our approach is based on that of Vaserstein in \cite{Vaserstein}, but we need to get round to technical difficulties arising from the function field setting. 

A simple but important idea is to prove that for a given matrix $\alpha \in \SL_2(\A)$, there exist matrices, say $L_{\alpha}, R_{\alpha} \in \SL_2(\A)$ that can be polynomially parameterized such that $L_{\alpha}\alpha R_{\alpha}$ belongs to a subset of $\SL_2(\A)$ that can be easily proved to belong to a polynomial family. Note that a product of matrices in $\SL_2(A)$, each of which comes from a polynomial family, also belongs to a polynomial family. With this remark, one can choose $L_{\alpha}, R_{\alpha}$ in such a way that each of them is a product of a fixed number of matrices, each of which belongs to a polynomial family. 

Let $\S_{PF}$ be the subset of $\SL_2(\A)$ consisting of all matrices $\begin{pmatrix} a & b \\ c & * \end{pmatrix} \in \SL_2(\A)$ that satisfy the following conditions:
\begin{itemize}

 \item [(i)] $a^{e_1} \equiv \epsilon_1 \pmod{b}$ for some unit $\epsilon_1 \in \bF_q^{\times}$ and some $e_1 \in \bZ_{>0}$;

\item [(ii)] $a^{e_2} \equiv \epsilon_2 \pmod{c}$ for some unit $\epsilon_2 \in \bF_q^{\times}$ and some $e_2 \in \bZ_{>0}$; and

\item [(iii)] $\gcd(e_1, e_2) = 1$.

\end{itemize}
 As shown in Corollary \ref{C---2nd----C----about---representation-of-alpha---in---S--SL2(A)}, the set $\S_{PF}$ belongs to a polynomial family, and in fact, the proof of this fact is the most difficult part in the proof of Theorem \ref{Thm-the-1st-main-thm-in-the-introduction}.

The structure of the polynomial ring $\A$ plays an important role when one wants to prove that for a given $\alpha \in \SL_2(\A)$, there exist elements $u, v \in \A$ such that
\begin{align*}
\begin{pmatrix} 1 & 0 \\ v & 1\end{pmatrix} \alpha \begin{pmatrix} 1 & u \\ 0 & 1\end{pmatrix} \in \S_{PF}.
\end{align*}
In order to show this, we use the $(q - 1)$-th power residue symbol, and a strong function field analogue of Dirichlet's theorem on primes in arithmetic progressions that is not available in the number field context. This function field analogue of the Dirichlet theorem assures that one can choose $u, v \in \A$ such that both $\wp_1 = au + b$ and $\wp_2 = av + c$ are primes in $\A$, and $\deg(\wp_1), \deg(\wp_2)$ are relative prime. The former is necessary for us to use the $(q - 1)$-th power residue symbol, and the latter is crucial to transform $\alpha$ into an element in $\S_{PF}$. This forms a main part of Lemma \ref{L-representation-of-alpha--in---SL2(A)}.

It remains to show that $\S_{PF}$ belongs to a polynomial family. For this purpose, the main difficulty is to show that for a given matrix $\alpha = \begin{pmatrix} a & b \\ c &  * \end{pmatrix} \in \S_{PF}$, there exists a matrix $\beta_{PF}$ that belongs to a polynomial family such that $\alpha^{r_1} = \beta_{PF}$ for some positive integer $r_1$. Condition (i) in the definition of $\S_{PF}$ plays a central role in proving this fact. If this can be done for $\alpha$, taking the transpose of $\alpha$ and using condition (ii) in the definition of $\S_{PF}$, one can also show that there exists a matrix $\gamma_{PF}$ that belongs to a polynomial family such that $(\alpha^T)^{r_2} = \gamma_{PF}$ for some positive integer $r_2$. Hence conjugating both sides of the last equation by $\begin{pmatrix} 0 & 1 \\ - 1 & 0 \end{pmatrix}$, one gets $\alpha^{-r_2} = \gamma^{\star}_{PF}$ for another matrix $\gamma^{\star}_{PF}$ that also belongs to a polynomial family. Using condition (iii), one can further choose $r_1, r_2 \in \bZ_{>0}$ such that $r_1 - r_2 = 1$, and hence 
\begin{align*}
\alpha = \alpha^{r_1}\alpha^{-r_2} = \beta_{PF}\gamma^{\star}_{PF},
\end{align*}
which proves that $\S_{PF}$ belongs to a polynomial family. This will be proved in detail in Corollary \ref{C---2nd----C----about---representation-of-alpha---in---S--SL2(A)}.

In order to prove that condition (i) in the definition of $\S_{PF}$ implies that for each matrix $\alpha$ in $\S_{PF}$, some power of $\alpha$ belongs to a polynomial family, the key step is to show that for each $\alpha =  \begin{pmatrix} a & b \\ c & d \end{pmatrix} \in \SL_2(\A)$ and each positive integer $r$, there exists a matrix $\gamma_{PF}$ that belongs to a polynomial family such that
\begin{align}
\label{Eqn1-in-main-ideas-of-the-proof}
\alpha^r = \begin{pmatrix} a^r & \epsilon b  \\ * & * \end{pmatrix} \gamma_{PF}
\end{align}
for some unit $\epsilon \in \bF_q^{\times}$. The first part of Lemma \ref{L-2nd-L-in-S-SL2(A)} shows that a simple use of the Caley-Hamilton theorem implies
\begin{align*}
\alpha^r = \begin{pmatrix} au + v & ub \\ * & * \end{pmatrix}
\end{align*}
for some $u, v \in \A$. One way to get (\ref{Eqn1-in-main-ideas-of-the-proof}) from the last equation is to transform the matrix $\begin{pmatrix} a + ub & ub \\ \star & \star \end{pmatrix}$ into the matrix $\begin{pmatrix} au + v & \epsilon b \\  * & * \end{pmatrix}$ for some unit $\epsilon \in \bF_q^{\times}$ by multiplying the former matrix by an appropriate matrix $\beta_{PF}$ that belongs to a polynomial family. This is proved in Lemma  \ref{L-1st-lemma-in-S-SL2(A)} whose proof also uses a strong function field analogue of the Dirichlet theorem and the $(q - 1)$-th power residue symbol. Note that for each $\alpha \in \SL_2(\A)$ and each positive integer $r$, the matrix $\gamma_{PF}$ in (\ref{Eqn1-in-main-ideas-of-the-proof}) is constructed as a product of a fixed number of matrices, each of which belongs to either the set of elementary matrices, the polynomial family $\cM_{\Lambda}$, or the polynomial family $\cM_{\Lambda}^{T}$. The last two polynomial families will be introduced in Subsection \ref{SS-Lambda-M_Lambda}.

The structure of this paper is as follows. In Section \ref{S-basic-notions}, we introduce some basic notation and necessary tools that will be used to prove Theorem \ref{Thm-the-1st-main-thm-in-the-introduction}. We will prove Theorem \ref{Thm-the-1st-main-thm-in-the-introduction} in Section \ref{S-SL2(A)-polynomial-family}.

\section{Some basic notation and notions}
\label{S-basic-notions}

In this section, we introduce some basic notation and notions that will be used throughout this paper. Vaserstein \cite{Vaserstein} used the polynomial matrices $\Phi_5, \Delta_i, \Gamma_i$ (see \cite[pages {\bf990}, {\bf992}]{Vaserstein} for their definitions) to construct the polynomial matrix in 46 variables that is a polynomial parametrization of $\SL_2(\bZ)$. We use the same set of polynomial matrices with different notation to obtain a polynomial parametrization of $\SL_2(\A)$; more explicitly, $\Lambda, \cF_i, \cG_i$ in this paper stand for $\Phi_5, \Delta_i, \Gamma_i$ in Vaserstein \cite{Vaserstein}, respectively. 

Note that the main aim of this section is to fix notation and notions for the next section. Hence the reader may wish to skip it on the first reading, and return to it later.

\subsection{Definitions of $\cF_h$, $\cG_h$}
\label{SS-F_h--and--G_h}

For each $m \in \A$, set $m_{\{1, 2\}} = \begin{pmatrix} 1 & m \\ 0 & 1 \end{pmatrix}$, and let $m_{\{2, 1\}} = \begin{pmatrix} 1 & 0 \\ m & 1 \end{pmatrix}$. Both $m_{\{1, 2\}}$ and $m_{\{2, 1\}}$ of course are in $\SL_2(\A)$. 

Although the following result is elementary, it is useful in many places of this paper. 

\begin{lemma}
\label{L-eqn-of-alpha-after-conjugating-by-0--1--minus1---0}

Let $\alpha \in \SL_2(\A)$. Then
\begin{align*}
(\alpha^{-1})^T = \begin{pmatrix} 0 & 1 \\ - 1 & 0 \end{pmatrix}^{-1} \alpha \begin{pmatrix} 0 & 1 \\ - 1 & 0 \end{pmatrix} =  \begin{pmatrix} 0 & 1 \\ - 1 & 0 \end{pmatrix} \alpha \begin{pmatrix} 0 & 1 \\ - 1 & 0 \end{pmatrix}^{-1}.
\end{align*}

\end{lemma}

For each $h \ge 1$, we denote by $\cF_{2h}(m^{(1)}, \ldots, m^{(2h)}) \in \SL_2(\A[m^{(1)}, \ldots, m^{(2h)}])$ the polynomial matrix in $2h$ parameters defined by
\begin{align*}
\cF_{2h}(m^{(1)}, \ldots, m^{(2h)}) &= m^{(1)}_{\{1, 2\}}m^{(2)}_{\{2, 1\}} \cdots m^{(2h - 1)}_{\{1, 2\}} m^{(2h)}_{\{2, 1\}} \\
&= \begin{pmatrix} 1 & m^{(1)} \\ 0 & 1 \end{pmatrix}\begin{pmatrix} 1 & 0 \\ m^{(2)} & 1 \end{pmatrix}\cdots \begin{pmatrix} 1 & m^{(2h - 1)} \\ 0 & 1 \end{pmatrix}\begin{pmatrix} 1 & 0 \\ m^{(2h)} & 1 \end{pmatrix}.
\end{align*}

For each $h \ge 0$, we denote by $\cF_{2h + 1}(m^{(1)}, \ldots, m^{(2h + 1)}) \in \SL_2(\A[m^{(1)}, \ldots, m^{(2h + 1)}])$ the polynomial matrix in $2h + 1$ parameters defined by
\begin{align*}
\cF_{2h + 1}(m^{(1)}, \ldots, m^{(2h + 1)}) &= m^{(1)}_{\{1, 2\}}m^{(2)}_{\{2, 1\}} \cdots m^{(2h)}_{\{2, 1\}} m^{(2h  + 1)}_{\{1, 2\}} \\
&= \begin{pmatrix} 1 & m^{(1)} \\ 0 & 1 \end{pmatrix}\begin{pmatrix} 1 & 0 \\ m^{(2)} & 1 \end{pmatrix}\cdots \begin{pmatrix} 1 & 0 \\ m^{(2h)} & 1 \end{pmatrix}\begin{pmatrix} 1 & m^{(2h +1)} \\ 0 & 1 \end{pmatrix}.
\end{align*}

Note that since $\SL_2(\A)$ is not boundedly generated by the elementary matrices, none of the $\cF_h$ is surjective over $\A$.

For each integer $r \ge 1$, set
\begin{align}
\label{E-eqn-of-G_r-----in-----terms-------of-----F_r---in-S-basic-notions}
\cG_r(m^{(1)}, \ldots, m^{(r)}) = \begin{pmatrix} 0 & 1 \\ - 1 & 0 \end{pmatrix} \cF_r(m^{(1)}, \ldots, m^{(r)})\begin{pmatrix} 0 & 1 \\ - 1 & 0 \end{pmatrix}^{-1}.
\end{align}
Equivalently, one can write
\begin{align}
\label{E-eqn----of-F_r-----in-----terms-------of-----G_r---in-S-basic-notions}
\cF_r(m^{(1)}, \ldots, m^{(r)}) = \begin{pmatrix} 0 & 1 \\ - 1 & 0 \end{pmatrix} \cG_r(m^{(1)}, \ldots, m^{(r)}) \begin{pmatrix} 0 & 1 \\ - 1 & 0 \end{pmatrix}^{-1}.
\end{align}

The next result follows immediately from Lemma \ref{L-eqn-of-alpha-after-conjugating-by-0--1--minus1---0}.

\begin{lemma}
\label{L-explicit-eqn-of-G_r----in----S---basic-notions}

\begin{itemize}

\item []

\item [(i)] For each integer $h \ge 1$,
\begin{align*}
\cG_{2h}(m^{(1)}, \ldots, m^{(2h)}) = (-m^{(1)})_{\{2, 1\}}(-m^{(2)})_{\{1, 2\}} \cdots (-m^{(2h - 1)})_{\{2, 1\}} (-m^{(2h)})_{\{1, 2\}}.
\end{align*}

\item [(ii)] For each integer $h \ge 0$,

\begin{align*}
\cG_{2h +1}(m^{(1)}, \ldots, m^{(2h + 1)}) = (-m^{(1)})_{\{2, 1\}}(-m^{(2)})_{\{1, 2\}} \cdots (-m^{(2h)})_{\{1, 2\}} (-m^{(2h  + 1)})_{\{2, 1\}}.
\end{align*}

\end{itemize}

\end{lemma}

For each positive integer $r$, set
\begin{align*}
\cG_r(\A^r) = \left\{\cG_r(a_1, a_2, \ldots, a_r) \; | \; (a_1, \ldots, a_r) \in \A^r \right\}.
\end{align*}

The next two lemmas are obvious.

\begin{lemma}
\label{L-relations-between-------F_r------in---------S-basic-notions}

\begin{itemize}

\item []

\item [(i)] $\cF_i(\A^i) \subset \cF_j(\A^j)$ for any $1 \le i < j$.

\item [(ii)] $\cF_{2h}(\A^{2h})\cF_r(\A^r) \subset \cF_{2h + r}(\A^{2h + r})$ for each integer $h \ge 1$ and each integer $r \ge 1$.

\item [(iii)] $\cF_{2h + 1}(\A^{2h + 1})\cF_r(\A^r) \subset \cF_{2h + r}(\A^{2h + r})$ for each integer $h \ge 0$ and each integer $r \ge 1$.

\end{itemize}

\end{lemma}

\begin{lemma}
\label{L-relations----between-G_r------in---------S-basic-notions}

\begin{itemize}

\item []

\item [(i)] $\cG_i(\A^i) \subset \cG_j(\A^j)$ for any $1 \le i < j$.

\item [(ii)] $\cG_{2h}(\A^{2h})\cG_r(\A^r) \subset \cG_{2h + r}(\A^{2h + r})$ for each integer $h \ge 1$ and each integer $r \ge 1$.

\item [(iii)] $\cG_{2h + 1}(\A^{2h + 1})\cG_r(\A^r) \subset \cG_{2h + r}(\A^{2h + r})$ for each integer $h \ge 0$ and each integer $r \ge 1$.

\end{itemize}

\end{lemma}

The matrices $\begin{pmatrix} \epsilon & 0 \\ 0 & \epsilon^{-1} \end{pmatrix}$ and $\begin{pmatrix} 0 & -\epsilon \\ \epsilon^{-1} & 0 \end{pmatrix}$ for any $\epsilon \in \bF_q^{\times}$ appear naturally in the proof of our main theorem. The next result shows that these matrices are contained in $\cG_4(\A^4) \cap \cF_4(\A^4)$ and $\cG_3(\A^3) \cap \cF_3(\A^3)$, respectively.

\begin{lemma}
\label{L--epsilon-matrix-in-G4(A^4)--cap--F4(A^4)----in---S----basic-notions}

Let $\epsilon \in \bF_q^{\times}$ be a unit in $\A$. Then
\begin{itemize}

\item [(i)]
\begin{align*}
\begin{pmatrix} \epsilon & 0 \\ 0 & \epsilon^{-1} \end{pmatrix} \in \cG_4(\A^4) \cap \cF_4(\A^4).
\end{align*}

\item [(ii)] 
\begin{align*}
\begin{pmatrix} 0 & -\epsilon \\ \epsilon^{-1} & 0 \end{pmatrix} \in \cG_3(\A^3) \cap \cF_3(\A^3).
\end{align*}

\end{itemize}

\end{lemma}

\begin{proof}

Part $(i)$ follows immediately by noting that
\begin{align*}
\begin{pmatrix} \epsilon & 0 \\ 0 & \epsilon^{-1} \end{pmatrix}  = \cG_4((\epsilon - 1)/\epsilon, -1, 1 - \epsilon, 1/\epsilon) = \cF_4(-\epsilon, \epsilon^{-1} - 1, 1, \epsilon - 1) \in \cG_4(\A^4) \cap \cF_4(\A^4).
\end{align*}

Since
\begin{align*}
\begin{pmatrix} 0 & -\epsilon \\ \epsilon^{-1} & 0 \end{pmatrix} = \cG_3(-\epsilon^{-1}, \epsilon, -\epsilon^{-1}) = \cF_3(-\epsilon, \epsilon^{-1}, -\epsilon) \in \cG_3(\A^3) \cap \cF_3(\A^3),
\end{align*}
we obtain the assertion in part $(ii)$.

\end{proof}

Combining Lemmas \ref{L-relations-between-------F_r------in---------S-basic-notions}, \ref{L-relations----between-G_r------in---------S-basic-notions}, and \ref{L--epsilon-matrix-in-G4(A^4)--cap--F4(A^4)----in---S----basic-notions}, we obtain the following result that we will need in the proof of our main theorem.

\begin{corollary}
\label{C-epsilon-matrix-times-G_r---in------S---basic-notions}

\begin{itemize}

\item []

\item [(i)] For any unit $\epsilon \in \bF_q^{\times}$ and any integer $r \ge 1$,
\begin{align*}
\begin{pmatrix} \epsilon & 0 \\ 0 & \epsilon^{-1} \end{pmatrix} \cG_r(\A^r) \subset \cG_{r + 4}(\A^{r + 4}).
\end{align*}

\item [(ii)] For any unit $\epsilon \in \bF_q^{\times}$ and any integer $r \ge 1$,
\begin{align*}
\begin{pmatrix} \epsilon & 0 \\ 0 & \epsilon^{-1} \end{pmatrix} \cF_r(\A^r) \subset \cF_{r + 4}(\A^{r + 4}).
\end{align*}

\end{itemize}

\end{corollary}

\subsection{Definition of $\Psi$}
\label{SS-Psi} 

In this subsection, we recall the notion of the polynomial matrix $\Phi_3$ in Vaserstein \cite[page {\bf989}]{Vaserstein} that will be denoted by $\Psi$ in this paper. 

Let $\Psi \in \SL_2(\A[m_1, m_2, m_3])$ be the polynomial matrix in three variables $m_1, m_2, m_3$ defined by
\begin{align}
\label{E-Psi}
\Psi(m_1, m_2, m_3) = \begin{pmatrix} 1 + m_1m_2m_3 & m_1^2m_3 \\ - m_2^2m_3 & 1 - m_1m_2m_3 \end{pmatrix} \in \SL_2(\A[m_1, m_2, m_3]).
\end{align}

Note that $\Psi(m_1, m_2, m_3)$ is unipotent in $\SL_2(\A[m_1, m_2, m_3])$ since $(1 + m_1m_2m_3) + (1 - m_1m_2m_3) = 2$. 

\begin{remark}

The following remark is due to the referee. Since $\Psi$ has rational integral coefficients, one can view $\Psi$ as a polynomial matrix in $\SL_2(\fR[m_1, m_2, m_3])$ for any commutative ring $\fR$ with $1$. When $\fR$ is a principal ideal domain, one can show that every unipotent matrix in $\SL_2(\fR)$ lies in the image of $\Psi$. Indeed, every unipotent matrix $\mathfrak{M} = \begin{pmatrix} a & b \\ c & d \end{pmatrix} \in \SL_2(\fR)$ satisfies $a + d = 2$ and $ad - bc = 1$. If we let $a = 1 + w$ for some $w \in \fR$, then 
\begin{align*}
\mathfrak{M} = \begin{pmatrix} 1 + w & b \\ c & 1 - w \end{pmatrix}.
\end{align*}
Since $ad - bc = 1$, we know that $w^2 = -bc$. Set $m_3 = \gcd(b, c)$. Then replacing $m_3$ by $\epsilon m_3$ for some unit $\epsilon \in \fR$, if necessary, one can write $w = m_1m_2m_3$, $b = m_1^2m_3$, and $c = -m_2^2m_3$ for some $m_1, m_2 \in \fR$. Thus 
\begin{align*}
\mathfrak{M} = \begin{pmatrix} 1 + m_1m_2m_3 & m_1^2m_3 \\ -m_2^2m_3 & 1 - m_1m_2m_3 \end{pmatrix},
\end{align*}
which lies in the image of $\Psi$.

\end{remark}

\subsection{Definitions of $\Gamma$ and $\cM_{\Gamma}$}
\label{SS-Gamma-M_Gamma}

In this subsection, we recall the notion of the polynomial matrix $\Phi_4$ in Vaserstein \cite[page {\bf989}]{Vaserstein} that will be denoted by $\Gamma$ in this paper. 

Let $\Gamma \in \SL_2(\A[m_1, m_2, m_3, m_4])$ be the polynomial matrix defined by
\begin{align*}
\Gamma(m_1, m_2, m_3, m_4) = \begin{pmatrix} 1 - m_2m_4 & m_2^2 \\ -m_4^2 & 1 + m_2m_4 \end{pmatrix}\begin{pmatrix}1 - m_1m_3 & m_1^2 \\ -m_3^2 & 1 + m_1m_3 \end{pmatrix}\begin{pmatrix} 1 - m_2m_4 & m_2^2 \\ -m_4^2 & 1 + m_2m_4 \end{pmatrix} \begin{pmatrix} 0 & - 1 \\ 1 & 0 \end{pmatrix}.
\end{align*}

Set
\begin{align*}
\cM_{\Gamma} = \{\alpha \alpha^T \; | \; \alpha \in \SL_2(\A) \} \subset \SL_2(\A).
\end{align*} 
Following the same arguments as in Vaserstein \cite[page {\bf989}]{Vaserstein} with $\cM_{\Gamma}, \Gamma$ in the roles of $X_4, \Phi_4$, respectively, one sees that $\cM_{\Gamma} \subset \Gamma(\A^4).$

\subsection{Definitions of $\Lambda$, $\cM_{\Lambda}$, and $\cM_{\Lambda}^T$}
\label{SS-Lambda-M_Lambda}

In this subsection, we recall the notions of $\Phi_5$ and $X_5$ in Vaserstein \cite[page {\bf990}]{Vaserstein} that will be denoted by $\Lambda$ and $\cM_{\Lambda}$, respectively in this paper. The polynomial matrix $\Lambda$ will play a central role in a polynomial parametrization of $\SL_2(\A)$. 

Let $\Lambda \in \SL_2(\A[m_1, m_2, m_3, m_4, m_5])$ be the polynomial matrix in five variables $m_1, m_2, m_3, m_4, m_5$ defined by
\begin{align*}
\Lambda(m_1, m_2, m_3, m_4, m_5) = \begin{pmatrix} m_5 & 0 \\ 0 & 1 \end{pmatrix} \Gamma(1 + m_1m_5, m_2m_5, m_3m_5, 1 + m_4m_5)\begin{pmatrix} m_5 & 0 \\ 0 & 1 \end{pmatrix}^{-1}.
\end{align*}

Let $\cM_{\Lambda}$ be the set of matrices defined by
\begin{align*}
\cM_{\Lambda} = \left\{\begin{pmatrix} 1 + ae & be^2 \\ c & 1 + de \end{pmatrix}\begin{pmatrix} 1 + ae & ce^2 \\ b & 1 + de \end{pmatrix} \; | \; \text{$a, b, c, d, e \in \A$ such that} \;\begin{pmatrix} 1 + ae & be^2 \\ c & 1 + de \end{pmatrix} \in \SL_2(\A) \right\}
\end{align*}
Following the same arguments as in Vaserstein \cite[page {\bf990}]{Vaserstein}, we get that
\begin{align}
\label{E--M_Lambda--subset--of---Lambda(A5)--in--SS-Lambda-M_Lambda}
\cM_{\Lambda} \subset \Lambda(\A^5) \subset \SL_2(\A).
\end{align}

Set 
\begin{align*}
\cM_{\Lambda}^{-1} &= \{\alpha^{-1} \; | \; \alpha \in \cM_{\Lambda}\}, \\
\cM_{\Lambda}^T &= \{\alpha^T \; | \; \alpha \in \cM_{\Lambda}\}, \\
\cM_{\Lambda}^{-1, T} &= \{\alpha^{-1} \; | \; \alpha \in \cM_{\Lambda}^T\}.
\end{align*}

The next result follows immediately from Lemma \ref{L-eqn-of-alpha-after-conjugating-by-0--1--minus1---0}.
\begin{lemma}
\label{L-M_Lambda==conjugation-of-M_Lambda}

\begin{itemize}

\item []

\item [(i)] $\cM_{\Lambda}^{-1} = \cM_{\Lambda}$, and $\cM_{\Lambda}^{-1, T} = \cM_{\Lambda}^T$.

\item [(ii)]
\begin{align*}
\cM_{\Lambda}^T = \left\{\begin{pmatrix} 0 & 1 \\ - 1 & 0 \end{pmatrix}^{-1} \alpha \begin{pmatrix} 0 & 1 \\ - 1 & 0 \end{pmatrix} \; | \; \alpha \in \cM_{\Lambda} \right\} = \left\{\begin{pmatrix} 0 & 1 \\ - 1 & 0 \end{pmatrix} \alpha \begin{pmatrix} 0 & 1 \\ - 1 & 0 \end{pmatrix}^{-1} \; | \; \alpha \in \cM_{\Lambda} \right\}.
\end{align*}

\item [(iii)]
\begin{align*}
\cM_{\Lambda} = \left\{\begin{pmatrix} 0 & 1 \\ - 1 & 0 \end{pmatrix}^{-1} \alpha \begin{pmatrix} 0 & 1 \\ - 1 & 0 \end{pmatrix} \; | \; \alpha \in \cM_{\Lambda}^T \right\} = \left\{\begin{pmatrix} 0 & 1 \\ - 1 & 0 \end{pmatrix} \alpha \begin{pmatrix} 0 & 1 \\ - 1 & 0 \end{pmatrix}^{-1} \; | \; \alpha \in \cM_{\Lambda}^T \right\}.
\end{align*}

\end{itemize}
\end{lemma}

We define the polynomial matrix $\Lambda^T \in \SL_2(\A[m_1, m_2, m_3, m_4, m_5])$ in five variables $m_1, m_2, m_3, m_4, m_5$ by
\begin{align*}
\Lambda^T(m_1, m_2, m_3, m_4, m_5) = \begin{pmatrix} 0 & 1 \\ -1 & 0 \end{pmatrix} \Lambda(m_1, m_2, m_3, m_4, m_5)\begin{pmatrix} 0 & 1 \\ -1 & 0 \end{pmatrix}^{-1}.
\end{align*}
Equation (\ref{E--M_Lambda--subset--of---Lambda(A5)--in--SS-Lambda-M_Lambda}) and Lemma \ref{L-M_Lambda==conjugation-of-M_Lambda}$(ii)$ imply that 
\begin{align}
\label{E--M_Lambda^T--subset--of--Lambda^T(A5)------SS--Lambda--M_Lambda}
\cM_{\Lambda}^T \subset \Lambda^T(\A^5).
\end{align}

\subsection{The $d$-th power residue symbol in $\A$}
\label{SS-the-d-th-power-residue-symbol-in-A}

In this subsection, we briefly recall the notion of the $d$-th power residue symbol. We refer the reader to Rosen \cite[Chapter {\bf3}]{Rosen} for a more complete account.  

Let $\wp$ be a prime in $\A$, and let $d$ be a positive divisor of $q - 1$. (Recall that $q$ is the number of elements in $\bF_q$.) If $m$ is an element in $\A$ such that $\wp$ does not divide $m$, then it is well-known (see Rosen \cite[pages {\bf23}, {\bf24}]{Rosen}) that there exists a unique element of $\bF_q^{\times}$, denoted by $\left(\dfrac{m}{\wp}\right)_d$, such that
\begin{align*}
m^{\dfrac{q^{\deg(\wp)}- 1}{d}} \equiv \left(\dfrac{m}{\wp}\right)_d \pmod{\wp}.
\end{align*}

If $m$ is an element in $\A$ such that $\wp$ divides $m$, we simply define $\left(\dfrac{m}{\wp}\right)_d = 0$. We call the symbol $\left(\dfrac{m}{\wp}\right)_d$ the \textit{$d$-th power residue symbol}.

\section{$\SL_2(\A)$ is a polynomial family}
\label{S-SL2(A)-polynomial-family}

In this section, we prove Theorem \ref{Thm-the-1st-main-thm-in-the-introduction}. Although our proof is based on the work of Vaserstein \cite{Vaserstein}, we need to introduce new ideas to overcome several technical difficulties arising in the function field setting. Vaserstein \cite{Vaserstein} used Dirichlet's theorem on primes in arithmetic progressions and the quadratic residue symbol in some auxiliary results to obtain a polynomial parametrization for $\SL_2(\bZ)$. We cannot use these tools in the function field setting. For the proof of Theorem \ref{Thm-the-1st-main-thm-in-the-introduction}, we instead exploit the $(q - 1)$-th power residue symbol, and an improved version of the function field analogue of Dirichlet's theorem that justifies the existence of many irreducible polynomials of a given degree $d$ in an arithmetic progression in $\A$, provided that $d$ is sufficiently large.

\begin{lemma}
\label{L-1st-lemma-in-S-SL2(A)}

Let $a, b, u \in \A$, and let $\alpha = \begin{pmatrix} 1 + au & bu \\ * & * \end{pmatrix} \in \SL_2(\A)$. Then there exist elements $m, n \in \A$, $\epsilon \in \bF_q^{\times}$, and $\beta \in \cM_{\Lambda}$ such that the matrix
\begin{align}
\label{E-the-matrix-in-1st-L-in-S-SL2(A)}
\alpha (um)_{\{1, 2\}}n_{\{2, 1\}}(-\wp u)_{\{1, 2\}}\beta(-\epsilon^{-1} u n)_{\{1, 2\}}(-\epsilon m)_{\{2, 1\}}
\end{align}
is of the form $\begin{pmatrix} * & * \\ \epsilon b & 1 + au \end{pmatrix}$, where $\wp = b + m(1 + au)$.

\end{lemma}

\begin{proof}

If $1 + au = 0$, letting $m = n = 0$, $\wp = b$, and $\epsilon = -u \in \bF_q^{\times}$, we see that Lemma \ref{L-1st-lemma-in-S-SL2(A)} follows immediately.

For the rest of the proof, suppose that $1 + au \ne 0$. Since $\det(\alpha) = 1$, we deduce that $1 + au$, $b$ are relatively prime in $\A$. Set
\begin{align}
\label{E-wp-in-1st-L-in-S-SL2(A)}
\wp= b + m(1 + au),
\end{align}
where $m$ will be determined shortly. By Rosen \cite[Theorem {\bf 4.8}]{Rosen}, we know that there are infinitely many elements $m$ in $\A$ such that for such an element $m$, the polynomial $\wp$ is a monic prime whose degree is congruent to $q - 2$ modulo $q - 1$ and greater than $\deg(b)$. Take such a monic prime $\wp$ of degree greater than $\deg(b)$ for some element $m \in \A$. We know that there is some integer $r$ such that
\begin{align}
\label{E-deg(wp)-in-1st-L-in-S-SL2(A)}
\deg(\wp) = q - 2 + (q - 1)r.
\end{align}

We now prove that there is an element $\epsilon \in \bF_q^{\times}$ such that 
\begin{align}
\label{E-a-equiv-epsilon-a1^(q-1)-mod-wp-in-1st-L-S-SL2(A)}
a \equiv \epsilon a_1^{q - 1} \pmod{\wp},
\end{align}
where $a_1$ is an element in $\A$. Indeed, denote by $\left(\dfrac{\cdot}{\wp}\right)_{q - 1}$ the $(q - 1)$-th power residue symbol (see Subsection \ref{SS-the-d-th-power-residue-symbol-in-A} for its definition). If $a \equiv 0 \pmod{\wp}$, then one can take $a_1 = 0$, and (\ref{E-a-equiv-epsilon-a1^(q-1)-mod-wp-in-1st-L-S-SL2(A)}) holds trivially.

If $a \not\equiv 0 \pmod{\wp}$, set
\begin{align}
\label{E-epsilon1-in-1st-L-in-S-SL2(A)}
\epsilon_1 = \left(\dfrac{a}{\wp}\right)_{q - 1} \in \bF_q^{\times}.
\end{align}

We see from \cite[Proposition {\bf3.2}]{Rosen} that 
\begin{align*}
\left(\dfrac{a\epsilon_1}{\wp}\right)_{q - 1} = \left(\dfrac{a}{\wp}\right)_{q - 1}\left(\dfrac{\epsilon_1}{\wp}\right)_{q - 1} = \epsilon_1 \left(\epsilon_1^{\dfrac{q- 1}{q - 1}\deg(\wp)}\right) = \epsilon_1^{(q - 1)(r + 1)} = 1,
\end{align*}
and it thus follows from \cite[Proposition {\bf3.1}]{Rosen} that there exists an element $a_1 \in \A$ such that $a \epsilon_1 \equiv a_1^{q - 1} \pmod{\wp}$. Now (\ref{E-a-equiv-epsilon-a1^(q-1)-mod-wp-in-1st-L-S-SL2(A)}) follows immediately by letting $\epsilon = \epsilon_1^{-1}$. 

By (\ref{E-a-equiv-epsilon-a1^(q-1)-mod-wp-in-1st-L-S-SL2(A)}), there exists an element $n \in \A$ such that 
\begin{align}
\label{E-a+n*wp-==-epsilon-a1^(q-1)-----in-1st-L-in-S-SL2(A)}
a + n\wp = \epsilon a_1^{q - 1}. 
\end{align}

Set
\begin{align}
\label{E-1st-eqn-of-lambda-in-1st-L-in--S--SL2(A)}
\lambda = \alpha (um)_{\{1, 2\}}n_{\{2, 1\}}(-\wp u)_{\{1, 2\}}.
\end{align}
We see from (\ref{E-wp-in-1st-L-in-S-SL2(A)}) and (\ref{E-a+n*wp-==-epsilon-a1^(q-1)-----in-1st-L-in-S-SL2(A)}) that
\begin{align}
\label{E-the-matrix-lambda-in-1st-L-in-S-SL2(A)}
\lambda = \alpha (um)_{\{1, 2\}}n_{\{2, 1\}}(-\wp u)_{\{1, 2\}} &= \begin{pmatrix} 1 + au & bu \\ * & * \end{pmatrix}\begin{pmatrix} 1 & mu \\ 0 & 1 \end{pmatrix}\begin{pmatrix} 1 & 0 \\ n & 1 \end{pmatrix}\begin{pmatrix} 1 & -\wp u  \\ 0 & 1 \end{pmatrix} \nonumber \\
&= \begin{pmatrix} 1 + au & \wp u \\ * & * \end{pmatrix}\begin{pmatrix} 1 & 0 \\ n & 1 \end{pmatrix}\begin{pmatrix} 1 & -\wp u  \\ 0 & 1 \end{pmatrix} \nonumber \\
&= \begin{pmatrix} 1 +  u \epsilon a_1^{q - 1} & \wp u \\ * & * \end{pmatrix}\begin{pmatrix} 1 & -\wp u \\ 0 & 1 \end{pmatrix} \nonumber \\
&= \begin{pmatrix} 1 +  u \epsilon a_1^{q - 1}  & -\wp u^2\epsilon a_1^{q - 1}\\ c & d \end{pmatrix},
\end{align}
where $c, d$ are some elements in $\A$.

By (\ref{E-1st-eqn-of-lambda-in-1st-L-in--S--SL2(A)}), and since $\alpha \in \SL_2(\A)$, we know that $\det(\lambda) = 1$, and thus (\ref{E-the-matrix-lambda-in-1st-L-in-S-SL2(A)}) tells us that
\begin{align*}
\lambda^{-1} = \begin{pmatrix} d  & \wp u^2\epsilon a_1^{q - 1}\\ -c & 1 +  u \epsilon a_1^{q - 1} \end{pmatrix}.
\end{align*}

Since $p$ is odd (recall that $p$ is the characteristic of $\bF_q$), one can write $q - 1 = 2q_1$ for some positive integer $q_1$, and thus $u^2a_1^{q - 1} = (u a_1^{q_1})^2$. Since $\det(\lambda) = \det(\lambda^{-1}) = 1$, we deduce that $d = 1 + d_1u a_1^{q_1}$ for some $d_1 \in \A$. Hence $\lambda^{-1}$ can be written in the form
\begin{align}
\label{E--1st-eqn--of--lambda^-(-1)--in-1st-L---in--S---SL2(A)}
\lambda^{-1} = \begin{pmatrix} 1 + d_1(u a_1^{q_1})  & \wp \epsilon (u a_1^{q _1})^2\\ -c & 1 +  (\epsilon a_1^{q_1})(ua_1^{q_1}) \end{pmatrix}.
\end{align}

Set
\begin{align}
\label{E-eqn-of-rho---in-1st-L-in----S--SL2(A)}
\rho = \begin{pmatrix} 1 + d_1(u a_1^{q_1})  & -c (u a_1^{q _1})^2\\ \epsilon\wp & 1 +  (\epsilon a_1^{q_1})(ua_1^{q_1}) \end{pmatrix}.
\end{align}
By (\ref{E-a+n*wp-==-epsilon-a1^(q-1)-----in-1st-L-in-S-SL2(A)}), one can write 
\begin{align*}
\rho =  \begin{pmatrix} *  & * \\ \epsilon \wp& 1 +  (a + n\wp)u \end{pmatrix}.
\end{align*}

By (\ref{E--1st-eqn--of--lambda^-(-1)--in-1st-L---in--S---SL2(A)}) and (\ref{E-eqn-of-rho---in-1st-L-in----S--SL2(A)}), we see that $\lambda^{-1}\rho \in \cM_{\Lambda}$, where $\cM_{\Lambda}$ is defined in Subsection \ref{SS-Lambda-M_Lambda}. Set
\begin{align}
\label{E-eqn-of-beta-in---1st-L---in---S--SL2(A)}
\beta = \lambda^{-1}\rho \in \cM_{\Lambda}.
\end{align}
We know that
\begin{align*}
\rho(-\epsilon^{-1}u n)_{\{1, 2\}} = \begin{pmatrix} *  & * \\ \epsilon\wp& 1 +  (a + n\wp)u \end{pmatrix} \begin{pmatrix} 1 & -\epsilon^{-1} u n \\ 0 & 1 \end{pmatrix}  = \begin{pmatrix} * & * \\ \epsilon \wp & 1 + au \end{pmatrix},
\end{align*}
and it thus follows from (\ref{E-wp-in-1st-L-in-S-SL2(A)}) that
\begin{align*}
\rho(-\epsilon^{-1} u n)_{\{1, 2\}}(-\epsilon m)_{\{2, 1\}} &= \begin{pmatrix} * & * \\ \epsilon \wp & 1 + au \end{pmatrix}\begin{pmatrix} 1 & 0 \\ -\epsilon m & 1 \end{pmatrix} \\
&= \begin{pmatrix} * & * \\ \epsilon(\wp - m(1 + au)) & 1 + au \end{pmatrix} \\
&= \begin{pmatrix} * & * \\ \epsilon b & 1 + au \end{pmatrix}.
\end{align*}

Lemma \ref{L-1st-lemma-in-S-SL2(A)} now follows immediately from (\ref{E-1st-eqn-of-lambda-in-1st-L-in--S--SL2(A)}) and (\ref{E-eqn-of-beta-in---1st-L---in---S--SL2(A)}).

\end{proof}

\begin{lemma}
\label{L-2nd-L-in-S-SL2(A)}

Let $\alpha = \begin{pmatrix} a & b \\ c & d \end{pmatrix} \in \SL_2(\A)$, and let $r$ be a positive integer. Then there exist $t^{(1)}, t^{(2)}, \ldots, t^{(10)} \in \A$, $\epsilon \in \bF_q^{\times}$, $\beta \in \cM_{\Lambda}$, and $\gamma \in \cM_{\Lambda}^T$ such that
\begin{align*}
\alpha^r t^{(1)}_{\{1, 2\}}t^{(2)}_{\{2, 1\}}t^{(3)}_{\{1, 2\}} \beta t^{(4)}_{\{1, 2\}}t^{(5)}_{\{2, 1\}}t^{(6)}_{\{1, 2\}}t^{(7)}_{\{2, 1\}}\gamma t^{(8)}_{\{2, 1\}}t^{(9)}_{\{1, 2\}}t^{(10)}_{\{2, 1\}} = \begin{pmatrix} a^r & \epsilon b \\ * & * \end{pmatrix}.
\end{align*}

\end{lemma}

\begin{remark}

In the proof of Lemma \ref{L-2nd-L-in-S-SL2(A)} below, we follow the same arguments as that of Vaserstein \cite[Lemma {\bf1.2}]{Vaserstein}.

\end{remark}

\begin{proof}

By the Cayley--Hamilton theorem, we know that $\alpha$ satisfies its characteristic equation, that is,
\begin{align*}
\alpha^2 + f\alpha + {\bf1}_2 = 0,
\end{align*}
where $f = -\text{Trace}(\alpha)$, and ${\bf1}_2 = \begin{pmatrix} 1 & 0 \\ 0 & 1 \end{pmatrix}$. From the above equation, it is not difficult to prove that $\alpha^r$ can be written in the form
\begin{align}
\label{E-eqn-of-alpha^r-in--2nd-L--in-S-SL2(A)}
\alpha^r = u\alpha + v{\bf1}_2 = \begin{pmatrix} au + v & ub \\ cu  & du + v \end{pmatrix}
\end{align}
for some elements $u, v \in \A$. We see that $1 = \det(\alpha)^r= \det(\alpha^r) \equiv \det(v{\bf1}_2) = v^2 \pmod{u}$, and thus $u$ divides $(v - 1)(v + 1)$. Therefore there exist $u_1, u_2 \in \A$ such that $v \equiv 1 \pmod{u_1}$, $v \equiv - 1 \pmod{u_2}$, and $u = u_1u_2$.

Since $v \equiv 1 \pmod{u_1}$, there exists an element $v_1 \in \A$ such that $v = 1 + u_1v_1$. We see that 
\begin{align*}
v + ua = (1 + u_1v_1) + u_1u_2a = 1 + (v_1 + u_2a)u_1, 
\end{align*}
and $ub = (u_2b)u_1$. Applying Lemma \ref{L-1st-lemma-in-S-SL2(A)} with $\alpha^r, v_1 + u_2a, u_2b, u_1$ in the roles of $\alpha, a, b, u$, respectively, we see from (\ref{E-eqn-of-alpha^r-in--2nd-L--in-S-SL2(A)}) that there exist 
$t^{(1)}, t^{(2)}, t^{(3)}, t^{(4)}, w^{(1)} \in \A$, $\epsilon_1 \in \bF_q^{\times}$, and $\beta \in \cM_{\Lambda}$ such that
\begin{align}
\label{E-eqn-of-rho----in--2nd-L-in---S-SL2(A)}
\rho = \alpha^rt^{(1)}_{\{1, 2\}}t^{(2)}_{\{2, 1\}}t^{(3)}_{\{1, 2\}}\beta t^{(4)}_{\{1, 2\}}w^{(1)}_{\{2, 1\}} = \begin{pmatrix} * & * \\ \epsilon_1u_2b & v + ua \end{pmatrix}.
\end{align}

Set
\begin{align}
\label{E-eqn-of-chi---in---2nd-L-in---S--SL2(A)}
\chi := -\begin{pmatrix} 0 & 1 \\ -1 & 0 \end{pmatrix} \rho \begin{pmatrix} 0 & 1 \\ -1 & 0 \end{pmatrix}^{-1} = \begin{pmatrix} - v - ua & \epsilon_1u_2b \\ * & * \end{pmatrix} \in \SL_2(\A).
\end{align}
Since $v \equiv - 1 \pmod{u_2}$, we see that $v = -1 + u_2v_2$ for some $v_2 \in \A$, and thus 
\begin{align*}
-v - ua = 1 - u_2v_2 - u_1u_2a = 1 + (-v_2 - u_1a)u_2. 
\end{align*}
Applying Lemma \ref{L-1st-lemma-in-S-SL2(A)} with $\chi, -v_2 - u_1a, \epsilon_1 b, u_2$ in the roles of $\alpha, a, b, u$, we deduce that there exist $w^{(2)}, t^{(6)}, t^{(7)}, t^{(8)}, t^{(9)} \in \A$, $\epsilon_2 \in \bF_q^{\times}$, and $\beta_1 \in \cM_{\Lambda}$ such that
\begin{align*}
\chi w^{(2)}_{\{1, 2\}}(-t^{(6)})_{\{2, 1\}}(-t^{(7)})_{\{1, 2\}}\beta_1(-t^{(8)})_{\{1, 2\}}(-t^{(9)})_{\{2, 1\}} = \begin{pmatrix} * & * \\ \epsilon_1\epsilon_2b& -v - ua \end{pmatrix}.
 \end{align*}

Negating both sides of the above equation, and conjugating them by $\begin{pmatrix} 0 & 1 \\ -1 & 0 \end{pmatrix}^{-1}$, we get from Lemma \ref{L-eqn-of-alpha-after-conjugating-by-0--1--minus1---0} and (\ref{E-eqn-of-chi---in---2nd-L-in---S--SL2(A)}) that
\begin{align}
\label{E-eqn-of-zeta-in-terms-rho----in-2nd-L---in--S-SL2(A)}
\rho(-w^{(2)})_{\{2, 1\}}t^{(6)}_{\{1, 2\}}t^{(7)}_{\{2, 1\}}\gamma t^{(8)}_{\{2, 1\}}t^{(9)}_{\{1, 2\}} = \begin{pmatrix} v + ua & \epsilon b \\ * & * \end{pmatrix},
\end{align}
where $\epsilon = \epsilon_1\epsilon_2$, and 
\begin{align*}
\gamma = \begin{pmatrix} 0 & 1 \\ -1 & 0 \end{pmatrix}^{-1}\beta_1 \begin{pmatrix} 0 & 1 \\ -1 & 0 \end{pmatrix}.
\end{align*}
Note that since $\beta_1 \in \cM_{\Lambda}$, Lemma \ref{L-M_Lambda==conjugation-of-M_Lambda} implies that $\gamma \in \cM_{\Lambda}^T.$

We know that
\begin{align*}
\alpha = \begin{pmatrix} a & b \\ c & d \end{pmatrix} \equiv \begin{pmatrix} a & 0 \\ * & * \end{pmatrix} \pmod{b},
\end{align*}
and it thus follows from (\ref{E-eqn-of-alpha^r-in--2nd-L--in-S-SL2(A)}) that
\begin{align*}
\begin{pmatrix} au + v & ub \\ * & * \end{pmatrix} = u\alpha + v{\bf1}_2 = \alpha^r \equiv \begin{pmatrix} a^r & 0 \\ * & * \end{pmatrix} \pmod{b}. 
\end{align*}
Therefore $au + v \equiv a^r \pmod{b}$. Since $\epsilon \in \bF_q^{\times}$ is a unit in $\A$, there exists an element $t^{(10)} \in \A$ such that
\begin{align*}
a^r = au + v + t^{(10)}\epsilon b.
\end{align*}
Hence we deduce from (\ref{E-eqn-of-zeta-in-terms-rho----in-2nd-L---in--S-SL2(A)}) that
\begin{align}
\label{E-2nd-eqn-of-zeta-in-terms-of-rho----in--2nd-L-in---S-SL2(A)}
\rho(-w^{(2)})_{\{2, 1\}}t^{(6)}_{\{1, 2\}}t^{(7)}_{\{2, 1\}}\gamma t^{(8)}_{\{2, 1\}}t^{(9)}_{\{1, 2\}}t^{(10)}_{\{2, 1\}} = \begin{pmatrix} v + ua & \epsilon b \\ * & * \end{pmatrix}\begin{pmatrix} 1 & 0 \\ t^{(10)} & 1 \end{pmatrix} = \begin{pmatrix} a^r & \epsilon b \\ * & * \end{pmatrix}.
\end{align}

Set
\begin{align*}
t^{(5)} = w^{(1)} - w^{(2)} \in \A,
\end{align*}
and note that 
\begin{align*}
(t^{(5)})_{\{2, 1\}} = (w^{(1)})_{\{2, 1\}}(-w^{(2)})_{\{2, 1\}}.
\end{align*}
Hence Lemma \ref{L-2nd-L-in-S-SL2(A)} follows immediately from (\ref{E-eqn-of-rho----in--2nd-L-in---S-SL2(A)}) and (\ref{E-2nd-eqn-of-zeta-in-terms-of-rho----in--2nd-L-in---S-SL2(A)}).

\end{proof}

\begin{lemma}
\label{L-3rd-L-in---S-SL2(A)}

Let $\alpha = \begin{pmatrix} a & b \\ c & d \end{pmatrix} \in \SL_2(\A)$. Let $\epsilon \in \bF_q^{\times}$, and let $r$ be a positive integer. Assume that 
\begin{align*}
a^r \equiv \epsilon \pmod{b}.
\end{align*}
Then there exist $t^{(1)}, t^{(2)}, \ldots, t^{(12)} \in \A$, $\beta \in \cM_{\Lambda}$, and $\gamma \in \cM_{\Lambda}^T$ such that
\begin{align*}
\alpha^rt^{(1)}_{\{1, 2\}}t^{(2)}_{\{2, 1\}}t^{(3)}_{\{1, 2\}} \beta t^{(4)}_{\{1, 2\}}t^{(5)}_{\{2, 1\}}t^{(6)}_{\{1, 2\}}t^{(7)}_{\{2, 1\}}\gamma t^{(8)}_{\{2, 1\}}t^{(9)}_{\{1, 2\}}t^{(10)}_{\{2, 1\}}t^{(11)}_{\{1, 2\}}t^{(12)}_{\{2, 1\}} = \begin{pmatrix} \epsilon & 0 \\ 0 & \epsilon^{-1} \end{pmatrix}.
\end{align*}

\end{lemma}

\begin{proof}

By Lemma \ref{L-2nd-L-in-S-SL2(A)}, there exist elements $\epsilon_1 \in \bF_q^{\times}$, $t^{(1)}, t^{(2)}, \ldots, t^{(9)} \in \A$, $w^{(1)} \in \A$, $\beta \in \cM_{\Lambda}$, and $\gamma \in \cM_{\Lambda}^T$ such that
\begin{align}
\label{E-eqn--of--rho--in--3rd---L-------in-S-SL2(A)}
\rho := \alpha^r t^{(1)}_{\{1, 2\}}t^{(2)}_{\{2, 1\}}t^{(3)}_{\{1, 2\}} \beta t^{(4)}_{\{1, 2\}}t^{(5)}_{\{2, 1\}}t^{(6)}_{\{1, 2\}}t^{(7)}_{\{2, 1\}} \gamma t^{(8)}_{\{2, 1\}}t^{(9)}_{\{1, 2\}}w^{(1)}_{\{2, 1\}} = \begin{pmatrix} a^r & \epsilon_1 b \\ * & * \end{pmatrix}.
\end{align}

By assumption, we know that $a^r \equiv \epsilon \pmod{b}$. Since $\epsilon_1 \in \bF_q^{\times}$ is a unit in $\A$, there exists an element $w^{(2)} \in \A$ such that
\begin{align*}
a^r + \epsilon_1b w^{(2)} = \epsilon,
\end{align*}
and thus  
\begin{align}
\label{E-2nd-eqn-of--rho---in--3rd---L--------in--S---SL2(A)}
\rho (w^{(2)})_{\{2, 1\}} =  \begin{pmatrix} a^r & \epsilon_1 b \\ * & * \end{pmatrix}\begin{pmatrix} 1 & 0 \\ w^{(2)} & 1 \end{pmatrix} = \begin{pmatrix} a^r + \epsilon_1bw^{(2)} & \epsilon_1b \\ * & * \end{pmatrix} = \begin{pmatrix} \epsilon & \epsilon_1b \\ * & * \end{pmatrix}.
\end{align}

Set
\begin{align*}
t^{(11)} = -\dfrac{\epsilon_1b}{\epsilon}.
\end{align*}
Since $\epsilon \in \bF_q^{\times}$ is a unit in $\A$, we get that $t^{(11)} \in \A$. We see from (\ref{E-2nd-eqn-of--rho---in--3rd---L--------in--S---SL2(A)}) that
\begin{align}
\label{E-3rd--eqn---of----rho----in-3rd---L-----in-S-SL2(A)}
\rho (w^{(2)})_{\{2, 1\}}(t^{(11)})_{\{1, 2\}} =\begin{pmatrix} \epsilon & \epsilon_1b \\ * & * \end{pmatrix}t^{(11)}_{\{1, 2\}} = \begin{pmatrix} \epsilon & \epsilon_1b \\ * & * \end{pmatrix}\begin{pmatrix} 1 & t^{(11)}\\ 0 & 1 \end{pmatrix} = \begin{pmatrix} \epsilon & \epsilon t^{(11)} +  \epsilon_1b \\ * & * \end{pmatrix} = \begin{pmatrix} \epsilon & 0 \\ m & n \end{pmatrix},
\end{align}
where $m, n$ are certain elements in $\A$. 

By (\ref{E-eqn--of--rho--in--3rd---L-------in-S-SL2(A)}), we know that $\det(\rho) = 1$, and thus
\begin{align*}
\epsilon n = \det\begin{pmatrix} \epsilon & 0 \\ m & n \end{pmatrix} = \det(\rho w^{(2)}_{\{2, 1\}}) = 1,
\end{align*}
and therefore $n = \epsilon^{-1}$. Hence (\ref{E-3rd--eqn---of----rho----in-3rd---L-----in-S-SL2(A)}) implies that
\begin{align}
\label{E-4th--eqn---of----rho----in---3rd---L-----in----S--SL2(A)}
\rho w^{(2)}_{\{2, 1\}} t^{(11)}_{\{1, 2\}} = \begin{pmatrix} \epsilon & 0 \\ m & \epsilon^{-1} \end{pmatrix}.
\end{align}

Set 
\begin{align*}
t^{(12)} = - \epsilon m \in \A.
\end{align*}
An easy calculation now shows that 
\begin{align}
\label{E-5th-----eqn---of-------rho----in---3rd---L-----in----S--SL2(A)}
\rho w^{(2)}_{\{2, 1\}} t^{(11)}_{\{1, 2\}}t^{(12)}_{\{2, 1\}} = \begin{pmatrix} \epsilon & 0 \\ m & \epsilon^{-1} \end{pmatrix} \begin{pmatrix} 1 & 0 \\ t^{(12)} & 1 \end{pmatrix} =\begin{pmatrix} \epsilon & 0 \\ 0 & \epsilon^{-1} \end{pmatrix}.
\end{align}
Setting
\begin{align*}
t^{(10)} = w^{(1)} + w^{(2)},
\end{align*}
we see that Lemma \ref{L-3rd-L-in---S-SL2(A)} follows immediately from (\ref{E-eqn--of--rho--in--3rd---L-------in-S-SL2(A)}) and (\ref{E-5th-----eqn---of-------rho----in---3rd---L-----in----S--SL2(A)}).

\end{proof}

\begin{corollary}
\label{C-1st-corollary--about---representation-of-alpha^r---in---S---SL2(A)}

Let $\alpha = \begin{pmatrix} a & b \\ c & d \end{pmatrix} \in \SL_2(\A)$. Let $\epsilon \in \bF_q^{\times}$, and let $r$ be a positive integer. Assume that 
\begin{align*}
a^r \equiv \epsilon \pmod{b}.
\end{align*}
Then
\begin{align*}
\alpha^r = \begin{pmatrix} \epsilon & 0 \\ 0 & \epsilon^{-1} \end{pmatrix}\chi_5 \gamma_{\Lambda} \chi_4 \beta_{\Lambda} \chi_3,
\end{align*}
where $\chi_3 \in \cF_3(\A^3), \chi_4 \in \cG_4(\A^4), \chi_5 \in \cG_5(\A^5)$, $\gamma_{\Lambda} \in \cM_{\Lambda}^T$, and $\beta_{\Lambda} \in \cM_{\Lambda}$.

\end{corollary}

\begin{proof}

By Lemma \ref{L-3rd-L-in---S-SL2(A)}, there exist $t^{(1)}, t^{(2)}, \ldots, t^{(12)} \in \A$, $\beta \in \cM_{\Lambda}$, and $\gamma \in \cM_{\Lambda}^T$ such that
\begin{align}
\label{E-1st-eqn-alpha^r-----in----1st--C----in----S----SL2(A)}
\alpha^rt^{(1)}_{\{1, 2\}}t^{(2)}_{\{2, 1\}}t^{(3)}_{\{1, 2\}} \beta t^{(4)}_{\{1, 2\}}t^{(5)}_{\{2, 1\}}t^{(6)}_{\{1, 2\}}t^{(7)}_{\{2, 1\}}\gamma t^{(8)}_{\{2, 1\}}t^{(9)}_{\{1, 2\}}t^{(10)}_{\{2, 1\}}t^{(11)}_{\{1, 2\}}t^{(12)}_{\{2, 1\}} = \begin{pmatrix} \epsilon & 0 \\ 0 & \epsilon^{-1} \end{pmatrix}.
\end{align}
We see that
\begin{align*}
\chi_5 = (t^{(8)}_{\{2, 1\}}t^{(9)}_{\{1, 2\}}t^{(10)}_{\{2, 1\}}t^{(11)}_{\{1, 2\}}t^{(12)}_{\{2, 1\}})^{-1} = (-t^{(12)}_{\{2, 1\}})(-t^{(11)})_{\{1, 2\}}(-t^{(10)})_{\{2, 1\}}(-t^{(9)})_{\{1, 2\}}(-t^{(8)})_{\{2, 1\}},
\end{align*}
and hence 
\begin{align}
\label{E-eqn-of-chi5--in---1st--C---in---S---SL2(A)}
\chi_5 = \cG_5(t^{(12)}, t^{(11)}, t^{(10)}, t^{(9)}, t^{(8)}) \in \cG_5(\A^5).
\end{align}

Similarly we see that
\begin{align}
\label{E-eqn----of----chi4--in---1st--C---in---S---SL2(A)}
\chi_4 = (t^{(4)}_{\{1, 2\}}t^{(5)}_{\{2, 1\}}t^{(6)}_{\{1, 2\}}t^{(7)}_{\{2, 1\}})^{-1} \in \cG_4(\A^4),
\end{align}
and
\begin{align}
\label{E-eqn----of------chi3--in---1st--C---in---S---SL2(A)}
\chi_3 = (t^{(1)}_{\{1, 2\}}t^{(2)}_{\{2, 1\}}t^{(3)}_{\{1, 2\}})^{-1} \in \cF_3(\A^3).
\end{align}

On the other hand, Lemma \ref{L-M_Lambda==conjugation-of-M_Lambda} implies that $\gamma_{\Lambda} = \gamma^{-1} \in \cM_{\Lambda}^T$, and $\beta_{\Lambda} = \beta^{-1} \in \cM_{\Lambda}$. It thus follows from (\ref{E-1st-eqn-alpha^r-----in----1st--C----in----S----SL2(A)}), (\ref{E-eqn-of-chi5--in---1st--C---in---S---SL2(A)}), (\ref{E-eqn----of----chi4--in---1st--C---in---S---SL2(A)}), and (\ref{E-eqn----of------chi3--in---1st--C---in---S---SL2(A)}) that
\begin{align*}
\alpha^r = \begin{pmatrix} \epsilon & 0 \\ 0 & \epsilon^{-1} \end{pmatrix}\chi_5 \gamma_{\Lambda} \chi_4 \beta_{\Lambda} \chi_3,
\end{align*}
where $\chi_3 \in \cF_3(\A^3), \chi_4 \in \cG_4(\A^4), \chi_5 \in \cG_5(\A^5)$, $\gamma_{\Lambda} \in \cM_{\Lambda}^T$, and $\beta_{\Lambda} \in \cM_{\Lambda}$ as desired.

\end{proof}

\begin{corollary}
\label{C---2nd----C----about---representation-of-alpha---in---S--SL2(A)}

Let $\alpha = \begin{pmatrix} a & b \\ c & d \end{pmatrix} \in \SL_2(\A)$. Assume that there exist relatively prime integers $r, s \ge 1$ such that $a^r \equiv \epsilon_1 \pmod{b}$ and $a^s \equiv \epsilon_2 \pmod{c}$ for some units $\epsilon_1, \epsilon_2 \in \bF_q^{\times}$. Then there exist $\chi_3 \in \cF_3(\A^3)$, $\chi_4 \in \cG_4(\A^4)$, $\chi_9 \in \cG_9(\A^9)$, $\chi_3^{\heartsuit} \in \cG_3(\A^3)$, $\chi_4^{\heartsuit} \in \cF_4(\A^4)$, $\chi_9^{\heartsuit} \in \cF_9(\A^9)$, $\gamma_{\Lambda}^{\heartsuit}, \beta_{\Lambda} \in \cM_{\Lambda}$, and $\gamma_{\Lambda}, \beta_{\Lambda}^{\heartsuit} \in \cM_{\Lambda}^T$ such that
\begin{align*}
\alpha =  \chi_9\gamma_{\Lambda} \chi_4 \beta_{\Lambda} \chi_3\chi_9^{\heartsuit} \gamma_{\Lambda}^{\heartsuit} \chi_4^{\heartsuit} \beta_{\Lambda}^{\heartsuit} \chi_3^{\heartsuit}.
\end{align*}

\end{corollary}

\begin{proof}

Since $r, s$ are relatively prime, one can find positive integers $h_1, h_2$ such that $sh_2 = rh_1 - 1$. By replacing $r, s$ by $rh_1, sh_2$, respectively, one can, without loss of generality, assume that $s = r - 1$.  

Applying Corollary \ref{C-1st-corollary--about---representation-of-alpha^r---in---S---SL2(A)}, one can write
\begin{align}
\label{E-1st-eqn-of-alpha^r---in---2nd---C---in---S---SL2(A)}
\alpha^r = \begin{pmatrix} \epsilon_1 & 0 \\ 0 & \epsilon_1^{-1} \end{pmatrix}\chi_5^{\#} \gamma_{\Lambda} \chi_4 \beta_{\Lambda} \chi_3,
\end{align}
where $\chi_3 \in \cF_3(\A^3), \chi_4 \in \cG_4(\A^4), \chi_5^{\#} \in \cG_5(\A^5)$, $\gamma_{\Lambda} \in \cM_{\Lambda}^T$, and $\beta_{\Lambda} \in \cM_{\Lambda}$.

Applying Corollary \ref{C-1st-corollary--about---representation-of-alpha^r---in---S---SL2(A)} with $\alpha^T$ in the role of $\alpha$, one can write
\begin{align}
\label{E-2nd----eqn---of-alpha^s---in---2nd---C---in---S---SL2(A)}
(\alpha^T)^s = \begin{pmatrix} \epsilon_2 & 0 \\ 0 & \epsilon_2^{-1} \end{pmatrix}\chi_5^* \gamma_{\Lambda}^* \chi_4^* \beta_{\Lambda}^* \chi_3^*,
\end{align}
where $\chi_3^* \in \cF_3(\A^3), \chi_4^* \in \cG_4(\A^4), \chi_5^* \in \cG_5(\A^5)$, $\gamma_{\Lambda}^* \in \cM_{\Lambda}^T$, and $\beta_{\Lambda}^* \in \cM_{\Lambda}$.

Conjugating both sides of (\ref{E-2nd----eqn---of-alpha^s---in---2nd---C---in---S---SL2(A)}) by $\begin{pmatrix} 0 & 1 \\ -1 & 0 \end{pmatrix}$, we deduce from Lemma \ref{L-eqn-of-alpha-after-conjugating-by-0--1--minus1---0} that
\begin{align}
\label{E---3rd---eqn----of-alpha^minus-s---in---2nd---C----in----S---SL2(A)}
\alpha^{-s} = \begin{pmatrix} \epsilon_2^{-1} & 0 \\ 0 & \epsilon_2 \end{pmatrix}\chi_5^{\diamond} \gamma_{\Lambda}^{\heartsuit} \chi_4^{\heartsuit} \beta_{\Lambda}^{\heartsuit} \chi_3^{\heartsuit},
\end{align}
where
\begin{align*}
\chi_5^{\diamond} &= \begin{pmatrix} 0 & 1 \\ -1 & 0 \end{pmatrix}\chi_5^* \begin{pmatrix} 0 & 1 \\ -1 & 0 \end{pmatrix}^{-1}, \\
\gamma_{\Lambda}^{\heartsuit} &= \begin{pmatrix} 0 & 1 \\ -1 & 0 \end{pmatrix}\gamma_{\Lambda}^* \begin{pmatrix} 0 & 1 \\ -1 & 0 \end{pmatrix}^{-1}, \\
 \chi_4^{\heartsuit} &= \begin{pmatrix} 0 & 1 \\ -1 & 0 \end{pmatrix}\chi_4^* \begin{pmatrix} 0 & 1 \\ -1 & 0 \end{pmatrix}^{-1}, \\
\beta_{\Lambda}^{\heartsuit} &= \begin{pmatrix} 0 & 1 \\ -1 & 0 \end{pmatrix}\beta_{\Lambda}^* \begin{pmatrix} 0 & 1 \\ -1 & 0 \end{pmatrix}^{-1}, \\
 \chi_3^{\heartsuit} &= \begin{pmatrix} 0 & 1 \\ -1 & 0 \end{pmatrix}\chi_3^* \begin{pmatrix} 0 & 1 \\ -1 & 0 \end{pmatrix}^{-1}.
\end{align*}
By Lemma \ref{L-M_Lambda==conjugation-of-M_Lambda} and equation (\ref{E-eqn----of-F_r-----in-----terms-------of-----G_r---in-S-basic-notions}) in Subsection \ref{SS-F_h--and--G_h}, one sees immediately that $\chi_3^{\heartsuit} \in \cG_3(\A^3), \chi_4^{\heartsuit} \in \cF_4(\A^4), \chi_5^{\diamond} \in \cF_5(\A^5)$, $\gamma_{\Lambda}^{\heartsuit} \in \cM_{\Lambda}$, and $\beta_{\Lambda}^{\heartsuit} \in \cM_{\Lambda}^T$. 

By Corollary \ref{C-epsilon-matrix-times-G_r---in------S---basic-notions}, 
\begin{align*}
\chi_9^{\heartsuit} :=  \begin{pmatrix} \epsilon_2^{-1} & 0 \\ 0 & \epsilon_2 \end{pmatrix} \chi_5^{\diamond} \in \cF_9(\A^9).
 \end{align*}
Similarly one sees that
\begin{align*}
\chi_9 :=  \begin{pmatrix} \epsilon_1 & 0 \\ 0 & \epsilon_1^{-1} \end{pmatrix} \chi_5^{\#} \in \cG_9(\A^9).
 \end{align*}

From (\ref{E-1st-eqn-of-alpha^r---in---2nd---C---in---S---SL2(A)}) and (\ref{E---3rd---eqn----of-alpha^minus-s---in---2nd---C----in----S---SL2(A)}), we deduce that
\begin{align*}
\alpha = \alpha^r\alpha^{-s} = \chi_9\gamma_{\Lambda} \chi_4 \beta_{\Lambda} \chi_3\chi_9^{\heartsuit} \gamma_{\Lambda}^{\heartsuit} \chi_4^{\heartsuit} \beta_{\Lambda}^{\heartsuit} \chi_3^{\heartsuit},
\end{align*}
which proves our contention.

\end{proof}

\begin{lemma}
\label{L-representation-of-alpha--in---SL2(A)}

Every element $\alpha \in \SL_2(\A)$ can be represented as
\begin{align*}
\alpha =  \chi_9\gamma_{\Lambda} \chi_4 \beta_{\Lambda} \chi_{11}^{\heartsuit} \gamma_{\Lambda}^{\heartsuit} \chi_4^{\heartsuit} \beta_{\Lambda}^{\heartsuit} \chi_4^{\#},
\end{align*}
where 
\begin{itemize}

\item [(i)] $\chi_4 \in \cG_4(\A^4)$, and $\chi_9 \in \cG_9(\A^9)$; 

\item [(ii)] $\chi_4^{\heartsuit} \in \cF_4(\A^4)$, and $\chi_{11}^{\heartsuit} \in \cF_{11}(\A^{11})$;

\item [(iii)] $\chi_4^{\#} \in \cG_4(\A^4)$;

\item [(iii)] $\gamma_{\Lambda}^{\heartsuit}, \beta_{\Lambda} \in \cM_{\Lambda}$, and $\gamma_{\Lambda}, \beta_{\Lambda}^{\heartsuit} \in \cM_{\Lambda}^T$.

\end{itemize}

\end{lemma}

\begin{proof}

Take any $\alpha = \begin{pmatrix} a & b \\ c & d \end{pmatrix} \in \SL_2(\A)$. We consider the following two cases:

$\star$ \textit{Case 1. $a = 0$.}

Since $\alpha \in \SL_2(A)$, we see that $b = -\epsilon$ and $c = \epsilon^{-1}$ for some unit $\epsilon \in \bF_q^{\times}$. One can write
\begin{align}
\label{E--eqn-of-alpha--in-the-case-when-a--==---0---in---main--L-----S---SL2(A)}
\alpha =  \begin{pmatrix} 0 & -\epsilon \\ \epsilon^{-1} & d \end{pmatrix} = \begin{pmatrix} 0 & -\epsilon \\ \epsilon^{-1} & 0 \end{pmatrix}(\epsilon d)_{\{1, 2\}} =  \chi_9\gamma_{\Lambda} \chi_4 \beta_{\Lambda} \chi_{11}^{\heartsuit} \gamma_{\Lambda}^{\heartsuit} \chi_4^{\heartsuit} \beta_{\Lambda}^{\heartsuit} \chi_4^{\#},
\end{align}
where
\begin{align*}
\chi_9 = \begin{pmatrix} 0 & -\epsilon \\ \epsilon^{-1} & 0 \end{pmatrix}(\epsilon d)_{\{1, 2\}},
\end{align*}
and
\begin{align*}
\gamma_{\Lambda} = \chi_4 = \beta_{\Lambda} = \chi_{11}^{\heartsuit} = \gamma_{\Lambda}^{\heartsuit} = \chi_4^{\heartsuit} = \beta_{\Lambda}^{\heartsuit} = \chi_4^{\#} = {\bf1}_2 = \begin{pmatrix} 1 & 0 \\ 0 & 1 \end{pmatrix}.
\end{align*}

Lemmas \ref{L--epsilon-matrix-in-G4(A^4)--cap--F4(A^4)----in---S----basic-notions}$(ii)$ and \ref{L-explicit-eqn-of-G_r----in----S---basic-notions} imply that 
\begin{align*}
\chi_9 = \begin{pmatrix} 0 & -\epsilon \\ \epsilon^{-1} & 0 \end{pmatrix}(\epsilon d)_{\{1, 2\}} \in \cG_4(\A^4),
\end{align*}
and it thus follows from Lemma \ref{L-relations----between-G_r------in---------S-basic-notions}$(i)$ that $\chi_9 \in \cG_9(\A^9)$. Lemma \ref{L-representation-of-alpha--in---SL2(A)} then follows immediately from (\ref{E--eqn-of-alpha--in-the-case-when-a--==---0---in---main--L-----S---SL2(A)}).

$\star$ \textit{Case 2. $a \ne 0$.}

By Rosen \cite[Theorem {\bf 4.8}]{Rosen}, there exist $u, v \in \A$ such that $au + b$, $av + c$ are primes and $\gcd(\deg(au + b), \deg(av + c)) = 1$. Set 
\begin{align*}
\wp_1 &= au + b, \\
\wp_2 &= av + c, \\
e_1 &= \dfrac{q^{\deg(\wp_1)} - 1}{q - 1}, \\
e_2 &= \dfrac{q^{\deg(\wp_2)} - 1}{q - 1}.
\end{align*}
The choice of $u, v$ implies that $\gcd(\deg(\wp_1), \deg(\wp_2)) = 1$.

We see that
\begin{align*}
\gcd(q^{\deg(\wp_1)} - 1, q^{\deg(\wp_2)} - 1) = q^{\gcd(\deg(\wp_1), \deg(\wp_2))} - 1 = q - 1,
\end{align*}
and thus
\begin{align}
\label{E--gcd--of--e1-and-e2-is-1--in--the-main-L-of---S--SL2(A)}
\gcd(e_1, e_2) = 1.
\end{align}

Set
\begin{align*}
\epsilon_1 = \left(\dfrac{a}{\wp_1}\right)_{q - 1} \in \bF_q^{\times}, \\
\epsilon_2 = \left(\dfrac{a}{\wp_2}\right)_{q - 1} \in \bF_q^{\times},
\end{align*}
where the $\left(\dfrac{\cdot}{\wp_i}\right)_{q - 1}$ denotes the $(q - 1)$-th power residue symbol. It is well-known (see Rosen \cite[Chapter 3]{Rosen} or Subsection \ref{SS-the-d-th-power-residue-symbol-in-A}) that 
\begin{align}
\label{E--1st--eqn--in--main--L--in--S---SL2(A)}
a^{e_1}  \equiv \epsilon_1 \pmod{\wp_1},
\end{align}
and
\begin{align}
\label{E--2nd---eqn--in--main--L--in--S---SL2(A)}
a^{e_2}  \equiv \epsilon_2 \pmod{\wp_2}.
\end{align}

We see that
\begin{align}
\label{E--3rd--eqn-in--main--L---S--SL2(A)}
v_{\{2, 1\}}\alpha u_{\{1, 2\}} = \begin{pmatrix} a & au + b \\ av + c & (av + c)u + bv + d \end{pmatrix} = \begin{pmatrix} a & \wp_1 \\ \wp_2 & (av + c)u + bv + d \end{pmatrix}.
\end{align}

Using (\ref{E--gcd--of--e1-and-e2-is-1--in--the-main-L-of---S--SL2(A)}), (\ref{E--1st--eqn--in--main--L--in--S---SL2(A)}), (\ref{E--2nd---eqn--in--main--L--in--S---SL2(A)}), and applying Corollary \ref{C---2nd----C----about---representation-of-alpha---in---S--SL2(A)} with $v_{\{2, 1\}}\alpha u_{\{1, 2\}}, e_1, e_2$ in the roles of $\alpha, r, s$, respectively, one can write
\begin{align*}
v_{\{2, 1\}}\alpha u_{\{1, 2\}} =  \chi_9^{\#}\gamma_{\Lambda} \chi_4 \beta_{\Lambda} \chi_3\chi_9^{\heartsuit} \gamma_{\Lambda}^{\heartsuit} \chi_4^{\heartsuit} \beta_{\Lambda}^{\heartsuit} \chi_3^{\heartsuit, \#},
\end{align*}
where $\chi_3 \in \cF_3(\A^3)$, $\chi_4 \in \cG_4(\A^4)$, $\chi_9^{\#} \in \cG_9(\A^9)$, $\chi_3^{\heartsuit, \#} \in \cG_3(\A^3)$, $\chi_4^{\heartsuit} \in \cF_4(\A^4)$, $\chi_9^{\heartsuit} \in \cF_9(\A^9)$, $\gamma_{\Lambda}^{\heartsuit}, \beta_{\Lambda} \in \cM_{\Lambda}$, and $\gamma_{\Lambda}, \beta_{\Lambda}^{\heartsuit} \in \cM_{\Lambda}^T$. The above equation implies that
\begin{align}
\label{E-4th---eqn--of--alpha--in--main--L--S---SL2(A)}
\alpha =  \chi_9\gamma_{\Lambda} \chi_4 \beta_{\Lambda}\chi_{11}^{\heartsuit} \gamma_{\Lambda}^{\heartsuit} \chi_4^{\heartsuit} \beta_{\Lambda}^{\heartsuit} \chi_4^{\#},
\end{align}
where
\begin{align*}
\chi_9 &= (-v)_{\{2, 1\}}\chi_9^{\#},\\
\chi_{11}^{\heartsuit} &= \chi_3\chi_9^{\heartsuit}, \\
\chi_4^{\#} &= \chi_3^{\heartsuit, \#} (-u)_{\{1, 2\}}.
\end{align*}

Since $\chi_3^{\heartsuit, \#} \in \cG_3(\A^3)$, the definition of $\cG_i$ and Lemma \ref{L-explicit-eqn-of-G_r----in----S---basic-notions} imply that $\chi_9 \in \cG_9(\A^9)$ and $\chi_4^{\#} \in \cG_4(\A^4)$. Furthermore Lemma \ref{L-relations-between-------F_r------in---------S-basic-notions} implies that $\chi_{11}^{\heartsuit} \in \cF_{11}(\A^{11})$. Hence Lemma \ref{L-representation-of-alpha--in---SL2(A)} follows from (\ref{E-4th---eqn--of--alpha--in--main--L--S---SL2(A)}).

\end{proof}

We now prove our main theorem in this paper.

\begin{theorem}
\label{T-SL2(A)--is--a--polynomial-family--in--S--SL2(A)}

$\SL_2(\A)$ is a polynomial family with 52 variables.

\end{theorem}

\begin{proof}

Let $\Omega$ be the polynomial matrix defined by
\begin{align*}
\Omega = \cG_9\Lambda^T\cG_4\Lambda \cF_{11}\Lambda\cF_4\Lambda^T\cG_4.
\end{align*}
We see that $\Omega$ has 52 variables. Using Lemma \ref{L-representation-of-alpha--in---SL2(A)}, and recalling that $\cM_\Lambda \subset \Lambda(\A^5)$ and $\cM_\Lambda^T \subset \Lambda^T(\A^5)$ (see Subsection \ref{SS-Lambda-M_Lambda}), we deduce that
\begin{align*}
\SL_2(\A) = \Omega(\A^{52}),
\end{align*}
which proves our contention.

\end{proof}

\section*{Acknowledgements}

I am grateful to the referee for an extremely careful reading of this paper, and suggesting insightful remarks and very useful comments on an earlier version of this paper. I thank Leonid Vaserstein for explaining some of his remarks in \cite{Vaserstein} to me. I would like to thank my parents, Nguyen Ngoc Quang and Phan Thi Thien Huong, for their constant support over the years.

\end{document}